\documentclass[11pt,reqno]{amsart}
\usepackage{amsmath,amssymb, xypic,verbatim,eucal, amscd,color,mathrsfs}
\usepackage{mathtools}
\usepackage{graphicx}
\usepackage{colordvi}
\usepackage{mathdots}
\usepackage{adjustbox}
\usepackage[colorlinks=true, pdfstartview=FitV,
linkcolor=cyan, citecolor=red, urlcolor=blue]{hyperref}
\usepackage{amsfonts,latexsym,bm}
\usepackage{tikz-cd}

\usepackage[all,arc]{xy}
\input xy

\numberwithin{equation}{section}

\theoremstyle{plain}
				
\newtheorem{theorem}{Theorem}[section]				
\newtheorem{proposition}[theorem]{Proposition}		
\newtheorem{corollary}[theorem]{Corollary}
\newtheorem{lemma}[theorem]{Lemma}

\theoremstyle{definition}

\newtheorem{definition}[theorem]{Definition}

\DeclareMathOperator{\ad}{ad}

\newcommand{\lra}{\longrightarrow}

\begin{document}

\title[Liouville symplectic form on the cotangent bundle of a
loop group]{Applications of the Liouville symplectic form on the
cotangent bundle of a loop group}
	
\author[I. Biswas]{Indranil Biswas}

\address{Department of Mathematics, Shiv Nadar University, NH91, Tehsil
Dadri, Greater Noida, Uttar Pradesh 201314, India}

\email{indranil.biswas@snu.edu.in, indranil29@gmail.com}

\author[M. Inaba]{Michi-aki Inaba}

\address{Department of Mathematical and Physical Sciences, 
Nara Women's University, Kitauoya-nishimachi, Nara 630-8506, Japan}

\email{inaba@cc.nara-wu.ac.jp}

\author[A. Komyo]{Arata Komyo}

\address{Department of Material Science, Graduate School of Science, University of Hyogo,
2167 Shosha, Himeji, Hyogo 671-2280, Japan}

\email{akomyo@sci.u-hyogo.ac.jp}

\author[S. Mukhopadhyay]{Swarnava Mukhopadhyay}

\address{School of Mathematics, Tata Institute of Fundamental Research,
Homi Bhabha Road, Mumbai 400005, India}

\email{swarnava@math.tifr.res.in}

\author[M.-H. Saito]{Masa-Hiko Saito}

\address{Department of Data Science, Faculty of Business Administration, 
Kobe Gakuin University, Minatojima, Chuou-ku, Kobe, 650-8586, Japan}

\email{mhsaito@ba.kobegakuin.ac.jp}

\subjclass[2010]{Primary: 14H60, 32G34, 53D50; Secondary: 81T40, 14F08}

\begin{abstract}
Let $G$ be a semisimple, simply connected, affine algebraic group defined over $\mathbb C$. Consider the Liouville symplectic
structure on the total space $T^*G((t))$ of the cotangent bundle of the loop group $G((t))$, where $t$ is a formal parameter. We show
that the Liouville symplectic structure on $T^*G((t))$ induces the symplectic structures on the moduli stack
of framed principal Higgs $G$--bundles on a compact connected Riemann surface $X$ and also on the moduli spaces of framed
$G$--connections on $X$. Similar symplectic structures --- on the moduli stack of framed
principal Higgs $G$--bundles, with finite order framing, and also framed connections on $X$, with finite order
framing --- were constructed earlier by various authors. Our results show that they all have a common origin.
\end{abstract}

\maketitle

\tableofcontents

\section{Introduction}

Let $X$ be a compact connected Riemann surface, or equivalently, an irreducible smooth projective curve defined over $\mathbb C$.
It's canonical line bundle will be denoted by $K_X$. Fix an effective divisor $\mathbb D$ on $X$. Let $G$ be a semisimple, simply connected, affine
algebraic group defined over $\mathbb C$. Take a principal $G$--bundle $E_G$ on the curve $X$, which is same as a holomorphic
principal $G$--bundle on the Riemann surface $X$. A Higgs field on $E_G$ is a section
$$
\theta \, \in\, H^0(X,\, \text{ad}(E_G)\otimes K_X\otimes {\mathcal O}_X({\mathbb D})),
$$
where $\text{ad}(E_G)\, \longrightarrow\, X$ is the 
adjoint vector bundle of $E_G$. A Higgs bundle is a principal $G$--bundle on $X$
equipped with a Higgs field. A Higgs bundle $(E_G,\, \theta)$ is called semistable (respectively, stable) if for
every pair $(P,\, \chi)$, where $P\, \subsetneq\, G$ is a proper parabolic subgroup and $\chi$ is a strictly anti--dominant character
of $P$ with respect to some Borel subgroup of $G$ contained in $P$ (this means that the line bundle on $G/P$ associated to $\chi$
is ample), and for every reduction of structure group $E_P\, \subset\, E_G$ to the subgroup $P$ such that
$$
\theta \, \in\, H^0(X,\, \text{ad}(E_P)\otimes K_X\otimes {\mathcal O}_X({\mathbb D}))
\,\subset\, H^0(X,\, \text{ad}(E_G)\otimes K_X\otimes {\mathcal O}_X({\mathbb D})),
$$
the inequality
$$
\text{degree}(E_P(\chi)) \, \geq\, 0 \ \ \, \text{(respectively, }\, \text{degree}(E_P(\chi)) \, >\, 0\text{)}
$$
holds, where $E_P(\chi)\, \longrightarrow\, X$ is the line bundle associated to the principal $P$--bundle $E_P$ for the
character $\chi$ of $P$.

One can construct a moduli space $M_{Higgs}(G)$ that parametrizes the stable principal Higgs $G$--bundles on $X$. When $\mathbb D\,=\,0$, this
moduli space $M_{Higgs}(G)$ carries a natural symplectic structure \cite{Hi2}. In the general case of $\mathbb D$, the moduli space
$M_{Higgs}(G)$ has a natural Poisson structure which was constructed in \cite{Bot}, \cite{Ma}.

A framed principal Higgs $G$--bundle is a Higgs bundle $(E_G,\, \theta)$ as above equipped with an enhancement given by a trivialization of
the principal $G$--bundle $E_G$ over the divisor $\mathbb D$. A framed Higgs bundle is called stable (respectively, semistable) if the underlying Higgs
bundle is stable (respectively, semistable). Let $\widetilde{M}_{Higgs}(G)$ denote the moduli space of stable framed Higgs bundles on $X$.
This moduli space $\widetilde{M}_{Higgs}(G)$ has a natural symplectic structure \cite{BLP1}, \cite{BLP2}. 
The natural projection $\widetilde{M}_{Higgs}(G)\, \longrightarrow\, M_{Higgs}(G)$ --- that simply
forgets the framing --- is
in fact a Poisson map. This result actually extends to the more general context
of parabolic Higgs bundles \cite{BLPS}.

Let $M_{Conn}(G)$ denote the moduli space of principal $G$--bundles $E_G$ on $X$ 
equipped with a meromorphic connection $\mathcal D$ whose pole is contained in the divisor 
$\mathbb D$ and satisfying an appropriate stability condition as done for the case of $M_{Higgs}(G)$.
Similar to the case of meromorphic Higgs bundles, in the special situaton where
$\mathbb D\,=\,0$, this moduli space of connections has a natural symplectic structure \cite{AB}, \cite{Go}. 
When $\mathbb{D}$ is a reduced divisor and $G\,=\,\mathrm{GL}(r,\mathbb{C})$,
the moduli space $M_{Conn}(G)$ was constructed in \cite{Nit}. This moduli space
$M_{Conn}(G)$ has a natural Poisson structure. The moduli space of parabolic connections 
(which is a parabolic version of the moduli space $M_{Conn}(G)$) 
was constructed in \cite{IIS1}, \cite{Inaba-1}. It was shown in \cite{IIS1}, \cite{Inaba-1}
that this moduli space has a natural symplectic structure. For a general $\mathbb D$, the 
moduli space $M_{Conn}(G)$ has a Poisson structure \cite{Boa1}, \cite{Boa2}.
Moduli spaces of connections with multiple pole are related to wild character
varieties via the generalized Riemann--Hilbert correspondence \cite{Boa1}, \cite{Boa2}.
From the point of view of the generalized Riemann--Hilbert correspondence,
it is reasonable to consider the moduli space of connections with parabolic structure,
which is constructed --- together with a symplectic structure ---
in \cite{IS} when we have $G\,=\,\mathrm{GL}(r,\mathbb{C})$.

One can enhance the structure of meromorphic connections using the notion of framed connections
just as it is done in the case for framed principal Higgs bundles. A framed connection is a pair
$(E_G,\, {\mathcal D})\, \in\, M_{Conn}(G)$ together with a trivialization of the principal $G$--bundle
$E_G$ over the divisor $\mathbb D$. The moduli space of framed connections $\widetilde{M}_{Conn}(G)$
 has a natural symplectic structure\cite{BIKS1}, \cite{BIKS2}. Moreover, the natural projection to
$M_{Conn}(G)$ from the moduli space $\widetilde{M}_{Conn}(G)$ of framed connections,
that simply forgets the framing, is actually a Poisson map.
Parabolic connections in \cite{IIS1}, \cite{Inaba-1} are related to framed connections.
Indeed, the parabolic structure of a parabolic connection 
is understood in terms of framing in \cite{BIKS1}, \cite{BIKS2}.
The symplectic structure on the moduli space of parabolic connections can be obtained from the
symplectic structure on $\widetilde{M}_{Conn}(G)$ by an argument involving reduction.

The non--abelian Hodge correspondence identifies the moduli space of Higgs bundles with the moduli space of 
principal bundles with algebraic connections (same as holomorphic connections) \cite{Si}, \cite{Hi1}, 
\cite{Do}, \cite{Co}. However, this identification is not algebraic or holomorphic but just $C^\infty$. 
Hence a symplectic structure on one side of the non--abelian Hodge correspondence does not automatically 
give rise to symplectic structures on the other side.

The main goal of this paper is to show that all the above symplectic structures have a single common origin. In fact,
they all originate from the Liouville 
symplectic structure on the total space $T^*G((t))$ of the cotangent bundle of the
loop group $$LG \ :=\ G((t)),$$ where $t$ is a formal parameter. Here we can think of the loop group $LG$
as the $\mathbb{C}$--valued points of a stack parametrizing the principal $G$--bundles $E_G$ on a curve $X$ with a given
trivialization of $E_G$ on a formal disk $D_p$ around $p\,\in\, X$ as 
well as a trivialization of $E_G$ on the punctured curve $X\backslash {p}$. We note that $G((t))$ is a Fr\'echet Lie group.

Throughout the rest of the paper, for the simplicity of the exposition, we will assume that the divisor above $\mathbb D$
on $X$ is a singleton, meaning $\mathbb{D}\,=\,\{p\}$, where $p$ is a point on 
$X$. The case of a general effective divisor $\mathbb{D}$ follows directly using the same methods.

Fix a point $p\, \in\, X$. Let $$L_XG \ :=\ G(X\backslash \{p\})$$ be the space of all algebraic maps from
$X\backslash \{p\}$ to the group $G$. Then the double quotient
$$
G[[t]]\backslash LG/L_XG
$$
is identified with the space of all $\mathbb{C}$--valued points of the moduli stack of principal $G$--bundles over $X$. Note that this
double quotient is constructed in the category of stacks. On the other hand, the quotient
$LG/L_XG$ is identified with the space of $\mathbb{C}$--valued points of the moduli stack of principal $G$--bundles $E_G$ over $X$ equipped
with a trivialization of $E_G$ over the formal completion $D_p$ of $X$ along $p$.

Let ${\mathcal M}_{Higgs}(G)$ denote the moduli stack of principal Higgs $G$--bundles $(E_G,\, \theta)$ on $X$ equipped with a trivialization of
the principal $G$--bundle $E_G$ over the formal completion $D_p$; the Higgs field $\theta$ on $E_G$ is allowed to have a pole at $p$
of arbitrary order, but it does not have any pole on the complement $X\setminus\{p\}$.
The right--translation of action of $L_XG$ on $LG$ produces an action of $L_XG$ on the total space
$T^*LG$ of the cotangent bundle of $LG$. The above moduli stack ${\mathcal M}_{Higgs}(G)$ is a quotient, by the action of $L_X(G)$, of a
certain $L_X(G)$--invariant subbundle
\begin{equation}\label{a1}
{\mathcal W}\ \subset T^*LG
\end{equation}
(see \eqref{e17} for the definition of $\mathcal W$).

We prove the following (see Theorem \ref{thm:main1}):

\begin{theorem}\label{thi0}
Restrict the Liouville symplectic form on $T^*LG$ to $\mathcal W$ in \eqref{a1}. This restriction descends to a
$2$--form on the quotient space ${\mathcal M}_{Higgs}(G)\,=\, {\mathcal W}/L_XG$. The $2$--form on ${\mathcal M}_{Higgs}(G)$ obtained this
way is actually a symplectic form.
\end{theorem}

Similar to Higgs bundles, we consider the following version of moduli space of connections. Let ${\mathcal M}_{Conn}(G)$ denote the moduli
stack of pairs $(E_G,\,\nabla)$, where $E_G$ principal $G$--bundles 
on $X$ and $\nabla$ is a meromorphic connection on $E_G$, such that the principal $G$--bundle $E_G$ is equipped with a trivialization of
over the formal completion $D_p$; the connection $\nabla$ is allowed to have a pole at a fixed point $p$ on $X$
of arbitrary order while the connection is regular on the complement $X\setminus \{p\}$. As
before, the right--translation of action of $L_XG$ on $LG$ produces an action of $L_XG$ on the total space
$T^*LG$ of the cotangent bundle of $LG$. The above moduli stack ${\mathcal M}_{Conn}(G)$ is a quotient, by the action of $L_X(G)$, of an
$L_X(G)$--invariant subbundle
\begin{equation}\label{a2}
{\mathcal U}\ \subset \ T^*LG
\end{equation}
(see \eqref{e31} for $\mathcal U$).

We prove the following (see Theorem \ref{thm:main2}):

\begin{theorem}\label{thi1}
Restrict the Liouville symplectic form on $T^*LG$ to $\mathcal U$ in \eqref{a2}. This restriction descends to a $2$--form on the
quotient space ${\mathcal M}_{Conn}(G)\,=\, {\mathcal U}/L_XG$. The $2$--form on ${\mathcal M}_{Conn}(G)$ obtained this
way is actually a symplectic form.
\end{theorem}

A key ingredient in the proofs of Theorem \ref{thi0} and Theorem \ref{thi1} is Theorem \ref{thm:residueann}. We will now
describe Theorem \ref{thm:residueann}; this theorem is in fact of independent interests.

To describe Theorem \ref{thm:residueann}, fix finitely many distinct points $\{Q_1,\, \cdots, \, Q_n\}\, \subset\, X$. Fix
a formal parameter $\xi_i$ at each $Q_i$, $1\, \leq\, i\, \leq\, n$. Take a principal $G$--bundle $E_G$ on $X$. Its adjoint
bundle will be denoted by $\ad(E_G)$. For each $1\, \leq\, i\, \leq\, n$, fix a trivialization of
$E_G$ on the formal completion $\widehat{Q}_i$ of $X$ along $Q_i$, which, in turn, gives a trivialization of the adjoint
bundle $\ad(E_G)$ on $\widehat{Q}_i$. Using these trivializations, we have
\begin{eqnarray}
H^0(X,\, \ad(E_G)(\ast \sum_{i=1}^nQ_i)) \,:=\, \lim_{j\rightarrow \infty} H^0(X,\, V\otimes{\mathcal O}_X(j\sum_{i=1}^n Q_i))
\ \hookrightarrow\ \bigoplus_{i=1}^n \mathfrak{g}\otimes \mathbb{C}((\xi_i)),\nonumber\\
H^0(X,\, \ad(E_G)\otimes K_X(\ast \sum_{i=1}^nQ_i))\,:=\,\lim_{j\rightarrow \infty}
H^0(X,\, \ad(E_G)\otimes K_X(j \sum_{i=1}^nQ_i))\ \hookrightarrow\ 
\bigoplus_{i=1}^n \mathfrak{g}\otimes \mathbb{C}((\xi_i))d\xi_i, \nonumber
\end{eqnarray}
where $\mathfrak g$ is the Lie algebra of $G$. There is a natural nondegenerate pairing 
$$
{\mathcal R}\,\, :\,\, (\bigoplus_{i=1}^n \mathfrak{g}\otimes \mathbb{C}((\xi_i)))\otimes
(\bigoplus_{i=1}^n\mathfrak{g}\otimes \mathbb{C}((\xi_i))d\xi_i)
\ \longrightarrow\ \mathbb{C}
$$
defined by $$((X_i\otimes f_i(\xi_i)_{i=1})^n),\,((Y_i\otimes g_i(\xi_i)d\xi_i)_{i=1}^n)\ \longmapsto\
\sum_{i=1}^n (X_i,Y_i)\operatorname{Res}_{\xi_i=0}(f_i(\xi_i)g_i(\xi_i)d\xi_i)$$
We refer the reader to equation \eqref{e3}.

Theorem \ref{thm:residueann}, which generalizes \cite[Theorem 1.22]{Uenobook}, says the following:

\begin{theorem}\label{thmi}
The subspace $$H^0(X,\, \ad(E_G)(\ast \sum_{i=1}^nQ_i))\ \subset\ \bigoplus_{i=1}^n \mathfrak{g}\otimes
\mathbb{C}((\xi_i))$$ and the subspace $$H^0(X, \, \ad(E_G)\otimes K_X(\ast \sum_{i=1}^nQ_i))\ \subset\
\bigoplus_{i=1}^n \mathfrak{g}\otimes \mathbb{C}((\xi_i))d\xi_i$$ are the annihilators of each other
under the residue pairing $\mathcal R$. 
\end{theorem}

Theorem \ref{thi0} also holds when then order of the pole at $\sum_{i=1}^n Q_i$ of the Higgs field is bounded by a positive
integer $m$ and the order of the infinitesimal neighborhood of $\sum_{i=1}^n Q_i$ on which $E_G$ is trivialized is $m-1$;
see Theorem \ref{th-j-f}.

Similar results are proved for the moduli stacks of framed connections; see Theorem \ref{th-j-f} and
Theorem \ref{thm:main3}. The proofs of these theorems follow along the same line as in the
case of Theorem \ref{thm:main1} and Theorem \ref{thm:main2}. A reason for this is that these stacks can be expressed as
sub--quotients of the cotangent bundle of the loop group $G((t))$. 

\section{Annihilators and the residue pairing}\label{se1}

In this section, we prove some general results about annihilators of adjoint bundle--valued forms on a smooth
complex curve under the residue pairing.

Let $X$ be an irreducible smooth complex projective curve or, equivalently, a compact connected Riemann
surface. The genus of $X$ will be denoted by $g$. The canonical line bundle of the curve $X$ will be
denoted by $K_X$.

Let $G$ be an affine algebraic semisimple group defined over $\mathbb C$; the Lie algebra of $G$ will be denoted by
$\mathfrak g$. Let $E_G$ be a principal $G$--bundle on the curve
$X$. The adjoint vector bundle for $E_G$,
which is the vector bundle on $X$ associated to the principal $G$--bundle $E_G$ for the adjoint action
of $G$ on its Lie algebra $\mathfrak g$, will be denoted by $\ad(E_G)$.

Fix finitely many distinct point $Q_1,\,\cdots,\, Q_n$ on $X$. For notational
convenience, for a vector bundle $V$ on $X$ and any integer $m$, the vector
bundle $$V\otimes {\mathcal O}_X(m\sum_{i=1}^n Q_i)\,=\, V\otimes ({\mathcal O}_X(\sum_{i=1}^n Q_i)^{\otimes m})
\, \longrightarrow\, X$$ will be denoted by $V(m\sum_{i=1}^n Q_i)$. The direct limit
$$
\lim_{j\rightarrow \infty} H^0(X,\, V(j\sum_{i=1}^n Q_i)),
$$
constructed using the natural inclusion maps $$H^0(X,\, V(j\sum_{i=1}^n Q_i))\, \hookrightarrow\, H^0(X,\, V((j+m)\sum_{i=1}^n Q_i)),$$
for $m\, \geq\, 0$, will be denoted by $H^0(X,\, V(\ast \sum_{i=1}^n Q_i))$.

For each $1\, \leq\, i\, \leq\, n$,
let $\xi_i$ be a choice of formal parameter at the point $Q_i$.
For each $1\, \leq\, i\, \leq\, n$, fix a trivialization of
$E_G$ on the formal completion $\widehat{Q}_i$ of $X$ along $Q_i$.
 So the restriction of
$\ad(E_G)$ to $\widehat{Q}_i$ is identified with the trivial Lie algebra bundle over $\widehat{Q}_i$ with
fiber $\mathfrak g$. Using these trivializations of $\ad(E_G)\big\vert_{\widehat{Q}_i}$, the
Laurent expansions via the formal parameters give us the following inclusion maps: 
\begin{eqnarray}\label{eqn:alurent}
\bigoplus_{i=1}^n{\iota}_i\,\,:&\,\,H^0(X,\, \ad(E_G)(\ast \sum_{i=1}^nQ_i))\ \,\hookrightarrow\ \, \bigoplus_{i=1}^n
\mathfrak{g}\otimes \mathbb{C}((\xi_i)),\label{e1}\\
\bigoplus_{i=1}^n{\eta}_i\,\,:&\,\, H^0(X,\, \ad(E_G)\otimes K_X(\ast \sum_{i=1}^nQ_i))\ \,\hookrightarrow\ \,
\bigoplus_{i=1}^n \mathfrak{g}\otimes \mathbb{C}((\xi_i))d\xi_i. \label{e2}
\end{eqnarray}

Let $(-\, ,\,-)$ denote the normalized Cartan--Killing form on the semisimple Lie algebra
$\mathfrak{g}$. This form induces the following residue pairing
\begin{equation}\label{e3}
{\mathcal R}\,\, :\,\, (\bigoplus_{i=1}^n \mathfrak{g}\otimes \mathbb{C}((\xi_i)))\otimes
(\bigoplus_{i=1}^n\mathfrak{g}\otimes \mathbb{C}((\xi_i))d\xi_i)
\ \longrightarrow\ \mathbb{C}
\end{equation}
$$
((X_i\otimes f_i(\xi_i)_{i=1})^n),\,((Y_i\otimes g_i(\xi_i)d\xi_i)_{i=1}^n)\ \longmapsto\
\sum_{i=1}^n (X_i,Y_i)\operatorname{Res}_{\xi_i=0}(f_i(\xi_i)g_i(\xi_i)d\xi_i).
$$
Note that $\operatorname{Res}_{\xi_i=0}(f_i(\xi_i)g_i(\xi_i)d\xi_i)$ is well--defined because the coefficients
of $\xi^k_i$ in the expansions of $f_i(\xi_i)$ and $g_i(\xi_i)$ vanish for all sufficiently negative $k$. Moreover the pairing $\mathcal{R}$ is nondegenerate.

The following theorem is a generalization of Theorem 1.22 in \cite{Uenobook}.

\begin{theorem}\label{thm:residueann}
The subspace $$H^0(X,\, \ad(E_G)(\ast \sum_{i=1}^nQ_i))$$ in \eqref{e1} and the subspace $$H^0(X, \,
\ad(E_G)\otimes K_X(\ast \sum_{i=1}^nQ_i))$$ in \eqref{e2} are the annihilators of each other
under the residue pairing $\mathcal R$ in \eqref{e3}. 
\end{theorem}

\begin{proof}
For any pair of positive integers $m$ and $N$,
consider the following short exact sequence of coherent sheaves on $X$: 
\begin{equation}\label{e4}
0\,\longrightarrow\, \ad(E_G)(-m \sum_{i=1}^n Q_i)\,\longrightarrow\, \ad(E_G)(N \sum_{i=1}^n Q_i)
\,\longrightarrow \,\bigoplus_{i=1}^n \bigoplus_{k= -N}^{m-1}\mathfrak{g}\otimes
\mathbb{C}\cdot\xi_i^k \,\longrightarrow\, 0;
\end{equation}
we have used the chosen trivializations of $\ad(E_G)$ over the formal completions $\widehat{Q}_i$, $1\, \leq\, i\, \leq\, n$, to
identify the quotient sheaf in \eqref{e4} with $\bigoplus_{i=1}^n \bigoplus_{k=-N}^{m-1}\mathfrak{g}\otimes
\mathbb{C}\xi_i^k$.
It can be shown that for $N$ sufficiently large,
\begin{equation}\label{j1}
H^1(X,\, \ad(E_G)(N \sum_{i=1}^n Q_i))\ =\ 0.
\end{equation}
For example, if we take any $N$ such that $Nn + \mu_{\rm min}(\ad (E_G))\, >\, 2(g-1)$, where
$\mu_{\rm min}(\ad (E_G))$ is the smallest one among the slopes of the successive quotients for the
Harder--Narasimhan filtration of the vector bundle $\ad(E_G)$, and $g\,=\, \text{genus}(X)$, then using Serre duality we have
$$
H^1(X,\, \ad(E_G)(N \sum_{i=1}^n Q_i))\,=\,
H^0(X,\, \ad(E_G)^*(-N \sum_{i=1}^n Q_i)\otimes K_X)^*\,=\,0,
$$
because the given condition that $Nn + \mu_{\rm min}(\ad (E_G))\, >\, 2(g-1)$ implies that we have
$$
\mu_{\rm max}(\ad(E_G)^*(-N \sum_{i=1}^n Q_i)\otimes K_X)\,=\, \mu_{\rm max}(\ad(E_G)^*) - Nn -N + 2(g-1)
$$
$$
\,=\, - \mu_{\rm min}(\ad(E_G)) - Nn -N + 2(g-1) \, <\, 0,$$ where
$\mu_{\rm max}(\ad (E_G)^*)$ is the largest one among the slopes of the successive quotients for the
Harder--Narasimhan filtration of the dual vector bundle $\ad(E_G)^*$. Here we are using the observation that any locally free coherent sheaf
on $X$, whose $\mu_{\rm max}$ is negative, does not admit any nonzero section; note that this observation follows immediately
from the fact that $\mu({\mathcal O}_X)\,=\, 0$. (See \cite{HL} for the construction and the properties of the
Harder--Narasimhan filtration.) Hence \eqref{j1} holds.

We note that the Cartan--Killing form on $\mathfrak g$ produces a fiberwise nondegenerate
pairing $\ad(E_G)\otimes \ad(E_G)\, \longrightarrow\, {\mathcal O}_X$, and hence we have an isomorphism
\begin{equation}\label{f1}
\ad(E_G)\ =\ \ad(E_G)^*.
\end{equation}
This implies that $\mu_{\rm max}(\ad (E_G))\,=\,- \mu_{\rm min}(\ad (E_G))$.

Next we observe that $H^0(X,\, \ad(E_G)(-m \sum_{i=1}^n Q_i))\,=\, 0$ for all $m$ sufficiently large.
Indeed, if $m$ is such that $\mu_{\rm max}(\ad (E_G))\,< \, mn$, then we have
$H^0(X,\, \ad(E_G)(-m \sum_{i=1}^n Q_i))\,=\, 0$ because
$$
\mu_{\rm max}(\ad (E_G)(-m \sum_{i=1}^n Q_i))\,= \,
\mu_{\rm max}(\ad (E_G)) - mn\,< \, 0.
$$

So take $N$ and $m$ to be sufficiently large such that
\begin{equation}\label{f2}
H^1(X,\, \ad(E_G)(N \sum_{i=1}^n Q_i))\,=\, 0\,=\, H^0(X,\, \ad(E_G)(-m \sum_{i=1}^n Q_i)).
\end{equation}

Consider the following long exact sequence of cohomologies corresponding to the short exact
sequence of sheaves in \eqref{e4}:
\begin{equation}
	\adjustbox{scale=.9,center}{
	\begin{tikzcd}
0 \ar[r]& H^0(X,\, \ad(E_G)(-m \sum_{i=1}^n Q_i)) \arrow[d, phantom, ""{coordinate, name=Z}]\ar[r]& 
H^0(X,\, \ad(E_G)(N \sum_{i=1}^n Q_i))
\ar[r,"\gamma"]& \bigoplus_{i=1}^n \bigoplus_{k= -N}^{m-1}\mathfrak{g}\otimes
\mathbb{C}\cdot\xi_i^k \arrow[dll,
"\delta"',
rounded corners,
to path={ -- ([xshift=2ex]\tikztostart.east)
	|- (Z) [near end]\tikztonodes
	-| ([xshift=-2ex]\tikztotarget.west)
	-- (\tikztotarget)}] \\
& H^1(X,\, \ad(E_G)(-m \sum_{i=1}^n Q_i)) \ar[r]& 
H^1(X,\, \ad(E_G)(N \sum_{i=1}^n Q_i)) \ar[r]& 0 .
\end{tikzcd}}
\end{equation}
Using \eqref{f2}, this reduces to the following short exact sequence:
\begin{equation}\label{e5}
\begin{tikzcd}
{} \arrow[r]
& 	H^0(X,\, \ad(E_G)(N \sum_{i=1}^n Q_i))\arrow[r,"\gamma"]
\arrow[d, phantom, ""{coordinate, name=Z}]
& \bigoplus_{i=1}^n \bigoplus_{k=-N}^{m-1}\mathfrak{g}\otimes \mathbb{C}\xi_i^k \arrow[dl,
"\delta"',
rounded corners,
to path={ -- ([xshift=2ex]\tikztostart.east)
|- (Z) [near end]\tikztonodes
-| ([xshift=-2ex]\tikztotarget.west)
-- (\tikztotarget)}] \\
& 	H^1(X,\, \ad(E_G)(-m \sum_{i=1}Q_i)) \arrow[r]
& 0.
\end{tikzcd}
\end{equation}

Next we observe that Serre duality, combined with the isomorphism given in \eqref{f1}, produces a perfect pairing 
\begin{equation}\label{f3}
B\,:\, H^1(X,\, \ad(E_G)(-m \sum_{i=1}^nQ_i))\otimes H^0(X,\, \ad(E_G)\otimes K_X(m \sum_{i=1}^nQ_i))\,
\longrightarrow\, \mathbb{C}.
\end{equation}
So this pairing $B$ identifies the cohomology
$H^1(X,\, \ad(E_G)(-m \sum_{i=1}^nQ_i))$ with the dual of $H^0(X, \ad(E_G)\otimes K_X(m \sum_{i=1}^nQ_i))$. 

Take $g(\xi_i)\,\in\, \bigoplus_{k=-N}^{m-1}\mathfrak{g}\otimes \mathbb{C}\xi_i^k$ for every
$1\, \leq\, i\, \leq\, n$. Consider
$$\delta((g(x_i)_{i=1}^n) \, \in\, H^1(X,\, \ad(E_G)(-m \sum_{i=1}Q_i)),$$ where $\delta$ is the homomorphism
in \eqref{e5}. One can show that
\begin{equation}\label{f4}
B(\delta((g(x_i))_{i=1}^n),\,\tau)\ =\ \mathcal{R}((g(x_i))_{i=1}^n,\, (\eta_i(\tau)_{i=1}^n))
\end{equation}
for all $\tau \,\in\, H^0(X,\, \ad(E_G)\otimes K_X(m \sum_{i=1}^nQ_i))$, where $\eta_i$ are the
homomorphisms in \eqref{e2} and $\mathcal R$ is the pairing in \eqref{e3}. Indeed, \eqref{f4}
follows immediately by comparing the constructions of $B$ and $\mathcal R$.

Since $B$ in \eqref{f3} is a perfect pairing, it follows immediately that we have $\delta((g(x_i))_{i=1}^n)
\,=\,0$ if and only if $$B(\delta((g(x_i))_{i=1}^n),\,\tau)\,=\,0$$ for all $\tau\,
\in \,H^0(X, \,\ad(E_G)\otimes K_X(m \sum_{i=1}^nQ_i))$. In view of \eqref{f4}, from this we
conclude that $\delta((g(x_i))_{i=1}^n)\,=\,0$ if and only if
$$\mathcal{R}((g(x_i))_{i=1}^n,\, ((\eta_i(\tau)_{i=1}^n))\,=\,0$$ for every $\tau \,\in\, H^0(X,\, 
\ad(E_G)\otimes K_X(m \sum_{i=1}^nQ_i))$.

Now by the exactness of equation \eqref{e5}, we have $\ker{\delta}\,=\,\operatorname{im}{\gamma}$. 
Therefore, we have proved the following lemma:
 
\begin{lemma}\label{lem:imp1}
An element $\omega\, \in\, \bigotimes_{j=1}^n \bigoplus_{k=-N}^{m-1}\mathfrak{g}\otimes
\mathbb{C}\xi_j^k$ lies in the subspace $$\gamma(H^0(X,\, \ad(E_G)(N\sum_{j=1}^nQ_j)))
\, \subset\, \bigoplus_{j=1}^n \bigoplus_{k=-N}^{m-1}\mathfrak{g}\otimes \mathbb{C}\xi_j^k$$
(see \eqref{e5} for $\gamma$) if and only if
$$
{\mathcal R}(\omega,\, \tau)\,\,=\,\, 0
$$
for every $\tau \,\in\, H^0 (C,\, \ad(E_G)\otimes K_X(m\sum_{j=1}^nQ_j))$,
where $\mathcal R$ is the residue pairing in \eqref{e3}; here $H^0 (C,\, \ad(E_G)\otimes
K_X(m\sum_{j=1}^nQ_j))$ is considered as a subspace of
$\bigoplus_{i=1}^n \mathfrak{g}\otimes \mathbb{C}((\xi_i))d\xi_i$ using the maps
$(\eta_i)_{i=1}^n$ in \eqref{e2}.
\end{lemma}

Lemma \ref{lem:imp1} will be used in completing the proof of Theorem \ref{thm:residueann}.

Suppose that an element $$\bm{\alpha}\,=\, (\bm \alpha_1,\, \cdots,\, \bm \alpha_n)
 \,\in\, \bigoplus_{i=1}^n \mathfrak{g}\otimes \mathbb{C}((\xi_i))$$ is 
annihilated by the subspace $H^0(X,\, \ad(E_G)\otimes K_X(\ast \sum_{i=1}^nQ_i))$ for the pairing $\mathcal R$
in \eqref{e3}; as before, $H^0(X,\, \ad(E_G)\otimes K_X(\ast \sum_{i=1}^nQ_i))$ is considered as a subspace
of $\bigoplus_{i=1}^n \mathfrak{g}\otimes \mathbb{C}((\xi_i))d\xi_i$ using the maps
$(\eta_i)_{i=1}^n$ in \eqref{e2}. We can write
\begin{equation}\label{f5}
\bm \alpha_i\ =\ \sum_{k=-N_i}^{\infty} a^{(i)}_k \xi_i^k,
\end{equation}
where $a_k^{(i)}\,\in\, \mathfrak{g}$.
Fix a positive integer $N'$ sufficiently large such that we have
\begin{equation}\label{f6}
H^1(X,\, \ad(E_G)((N'+j)\sum_{i=1}^nQ_i))\ = \ 0
\end{equation}
for all $j\, \geq\, 0$; it was observed earlier that it is possible to choose such an integer $N'$ (see \eqref{j1}). 

Now define
\begin{equation}\label{g1}
N\ :=\ \max\{N_1,\, \cdots, \, N_n;\, N'\},
\end{equation}
where $N_i$ are as in \eqref{f5} and $N'$ is the integer in \eqref{f6}. 

Recall the above condition on $\bm{\alpha}$ that it is annihilated by $H^0(X,\, \ad(E_G)\otimes
K_X(\ast \sum_{i=1}^nQ_i))$ for the pairing $\mathcal R$. For any
\begin{equation}\label{f10}
\omega_0\, \in\, H^0(X,\,\ad(E_G)\otimes K_X(\ast \sum_{i=1}^nQ_i)),
\end{equation}
write, using \eqref{e2},
\begin{equation}\label{f9}
\omega_0 \ =\ \bigoplus_{i=1}^n \sum_{j=-M_i}^{+\infty} X_{i,j}\otimes
\xi_i^j\cdot d\xi_i,
\end{equation}
where $X_{i,j}\, \in\, \mathfrak{g}$. Then we have
\begin{equation}\label{f7}
{\mathcal R}(\bm{\alpha},\, \omega_0)\,=\, \sum_{i=1}^n\sum_{k=-N_i}^{M_i-1} (a^{(i)}_k,\, X_{i, -(k+1)}). 
\end{equation}

For any fixed $m$ big enough, and $1\, \leq\, i\,\leq\, n$, we define a truncation $\bm{\alpha}_{i,m}$ of $\bm{\alpha_i}$ as follows:
\begin{equation}\label{g4}
\bm{\alpha}_{i,m}\ :=\ \sum_{k=-N_i}^{m-1} a^{(i)}_k \xi_i^k
\end{equation}
(see \eqref{f5}). Now consider $(\bm{\alpha}_{i,m})_{i=1}^n$. Note that for 
\begin{equation}\label{f8}
m\,> \,\max\{M_1,\,\cdots,\, M_n\},
\end{equation}
the section $\omega_0$ in \eqref{f10} lies in the following subspace
$$
\omega_0\ \in\ H^0(X,\ad(E_G)\otimes K_X(m\sum_{i=1}^nQ_i))\, \subset\,
H^0(X,\,\ad(E_G)\otimes K_X(\ast \sum_{i=1}^nQ_i)).
$$

Assume that $m$ satisfies the inequality in \eqref{f8}. It can be shown that
\begin{equation}\label{f11}
{\mathcal R}(\bm{\alpha},\, \omega_0)\ =\ {\mathcal R}(\bm{\alpha}_{i,m},\, \omega_0).
\end{equation}
To see this, first note that $\bm \alpha_i - \bm{\alpha}_{i,m}$ has a zero at each $Q_i$, $1\, \leq\, i\, \leq\, n$, of order at least
$m$. From the inequality in \eqref{f8} and the expression in \eqref{f9} we conclude that the pairing
$(\bm \alpha_i - \bm{\alpha}_{i,m}, \, \omega_0)\, \in\, \mathbb{C}((\xi_i))$ does note have any pole at $Q_i$.
This immediately implies that \eqref{f11} holds.

Since $\bm{\alpha}$ is annihilated by $H^0(X,\, \ad(E_G)\otimes
K_X(\ast \sum_{i=1}^nQ_i))$ for the pairing $\mathcal R$, from \eqref{f11} it follows that
$(\bm{\alpha}_{i,m})_{i=1}^n$ is also annihilated by $H^0(X,\,\ad(E_G)\otimes K_X(m\sum_{i=1}^nQ_i))$
for the pairing $\mathcal R$.

As $(\bm{\alpha}_{i,m})_{i=1}^n$ is annihilated by $H^0(X,\,\ad(E_G)\otimes K_X(m\sum_{i=1}^nQ_i))$
for the pairing $\mathcal R$, we observe that Lemma \ref{lem:imp1} implies that there is a global section
\begin{equation}\label{g2}
\bm{\alpha}^{(m)}\ \in\ H^0(X,\,\ad(E_G)(N\sum_{i=1}^nQ_i))
\end{equation}
(see \eqref{g1} for $N$) whose Laurent expansion at each $Q_i$ gives
$\bm{\alpha}_{i,m}$. Now further $\bm{\alpha}^{(m)}$ is also an element of the space
$H^0(X,\,\ad(E_G)(\ast \sum_{i=1}^nQ_i))$. So for any
$$
\widetilde{\omega}\ \in\ H^0(X,\,\ad(E_G)\otimes K_X(\ast \sum_{i=1}^nQ_i)),
$$
the pairing $(\bm{\alpha}^{(m)},\, \widetilde{\omega})$ is a meromorphic $1$--form on $X$.
Consequently, the total residue of the form $(\bm{\alpha}^{(m)},\, \widetilde{\omega})$ is zero.
Therefore, $\bm{\alpha}^{(m)}$ in \eqref{g2} is annihilated by any element of 
$H^0(X,\,\ad(E_G)\otimes K_X(\ast \sum_{i=1}^nQ_i))$ under the residue pairing. 

Hence we conclude that $(\bm{\alpha}_{i,m})_{i=1}^n$ is annihilated by
the entire $H^0(X,\,\ad(E_G)\otimes K_X(\ast \sum_{i=1}^nQ_i))$ under the pairing $\mathcal R$.

Now for each $i\,\in\,\{1,\,\cdots,\, n\}$, define
\begin{equation}\label{e02}
\bm{\beta}_i\ :=\ \bm{\alpha}_{i}-\bm{\alpha}^{(m)}.
\end{equation}
Observe that 
for all $i\,\in\, \{1,\,\cdots,\, n\}$, this $\bm{\beta}_i$ has a zero at $Q_i$ of order at least $m$
(see \eqref{f5} and \eqref{g4}).

To complete the proof of the theorem, it suffices to show that
\begin{equation}\label{g5}
\bm{\beta}_{i}\ =\ 0
\end{equation}
for every $1\, \leq\, i\, \leq\, n$.

Suppose for some $k$, we have
\begin{equation}\label{g6}
\bm{\beta}_k\ \not=\, 0.
\end{equation}
Thus we can find $s\,\geq\,m$ such that
\begin{equation}\label{e01}
\bm{\beta}_k\ =\ b_s\xi_k^{s}+b_{s+1}\xi^{s+1}_k+\dots ,
\end{equation}
with $b_s\, \not=\, 0$.

Recall that $\bm{\alpha}$ and $(\bm{\alpha}_{i,m})_{i=1}^n$ are both actually annihilated by
the entire $H^0(X,\,\ad(E_G)\otimes K_X(\ast \sum_{i=1}^nQ_i))$. Hence from \eqref{e02} it
follows that
\begin{equation}\label{e03}
{\mathcal R}((\bm{\beta}_i)_{i=1}^n,\, \omega_0) \ = \ 0
\end{equation}
for all $\omega_0\, \in\, H^0(X,\,\ad(E_G)\otimes K_X(\ast \sum_{i=1}^nQ_i))$.

Choose $m$ such that
$$
m + \mu_{\rm min}(\ad(E_G)) \ > \ 0.
$$
Since $s\, \geq\, m$ (see \eqref{e01}), this implies that
\begin{equation}\label{g7}
s + \mu_{\rm min}(\ad(E_G)) \ > \ 0.
\end{equation}
Consider the following short exact sequence of coherent sheaves on $X$:

$$
	\begin{tikzcd}
		0 \arrow[r]
		& 	\ad(E_G)\otimes K_X\otimes {\mathcal O}_X(s Q_k) \arrow[r,]
		\arrow[d, phantom, ""{coordinate, name=Z}]
		& \ad(E_G)\otimes K_X\otimes {\mathcal O}_X ((s+1)Q_k) \arrow[dl, 
		rounded corners,
		to path={ -- ([xshift=2ex]\tikztostart.east)
			|- (Z) [near end]\tikztonodes
			-| ([xshift=-2ex]\tikztotarget.west)
			-- (\tikztotarget)}] \\
		& 	\ad(E_G)_{Q_k}\otimes (K_X\otimes {\mathcal O}_X ((s+1)Q_k))_{Q_k} \arrow[r]
		& 0,
	\end{tikzcd}
$$
where $k$ is as in \eqref{e01} and $\ad(E_G)_{Q_k}\otimes (K_X\otimes {\mathcal O}_X ((s+1)Q_k))$
is the fiber of $\ad(E_G)_{Q_k}\otimes (K_X\otimes {\mathcal O}_X ((s+1)Q_k))$ over the point $Q_k$.
It gives an exact sequence of cohomologies
\begin{equation}\label{g8}
\begin{tikzcd}
	{} \arrow[r]
	& H^0(X,\ad(E_G)\otimes K_X\otimes {\mathcal O}_X((s+1) Q_k)) \arrow[r,"\rho"]
	\arrow[d, phantom, ""{coordinate, name=Z}]
	& \ad(E_G)\otimes K_X\otimes {\mathcal O}_X ((s+1)Q_k)_{Q_k} \arrow[dl,"\delta"',
	rounded corners,
	to path={ -- ([xshift=2ex]\tikztostart.east)
		|- (Z) [near end]\tikztonodes
		-| ([xshift=-2ex]\tikztotarget.west)
		-- (\tikztotarget)}] \\
	& 	H^1(X,\ad(E_G)_{Q_k}\otimes (K_X\otimes {\mathcal O}_X ((s+1)Q_k))) \arrow[r]
	& {}.
\end{tikzcd}
\end{equation}

By Serre duality,
$$
H^1(X,\, \ad(E_G)\otimes K_X\otimes {\mathcal O}_X(s Q_k))\,=\,
H^0(X,\, \ad(E_G)^*\otimes {\mathcal O}_X(-s Q_k))^*.
$$
We have
$$
\mu_{\rm max}(\ad(E_G)^*\otimes {\mathcal O}_X(-s Q_k))\,=\, - s - \mu_{\rm min}(\ad(E_G))\, <\, 0
$$
using \eqref{g7}. This implies that we have $H^0(X,\, \ad(E_G)^*\otimes {\mathcal O}_X(-s Q_k))\,=\, 0$, and hence
it follows that $$H^1(X,\, \ad(E_G)\otimes K_X\otimes {\mathcal O}_X(s Q_k))\,=\,0.$$ Consequently, the
homomorphism $\rho$ in \eqref{g8} is surjective.

Since $\rho$ in \eqref{g8} is surjective, there is a section
$$\omega' \ \in\ H^0(X,\, \ad(E_G)\otimes K_X\otimes {\mathcal O}_X ((s+1)Q_k))
\ \subset \ H^0(X,\, \ad(E_G)\otimes K_X((s+1)\sum_{i=1}^n Q_i))
$$
with the following property: Let $\omega'_{s+1}$ be the coefficient of $\xi^{-(s+1)}_k$ in the Laurent expansion of
$\omega'$ around the point $Q_k$ (see \eqref{g6} for $k$); then
\begin{equation}\label{g9}
(b_s,\, \omega'_{s+1}) \ \neq\ 0,
\end{equation}
where $b_s$ is as in \eqref{e01} (recall that $b_s\, \not=\, 0$).

Since $\omega'$, as a meromorphic section of $\ad(E_G)\otimes K_X$, has pole only at $Q_k$, and
$(\bm{\beta}_i)_{i=1}^n$ has no pole at any $Q_i$, we conclude that
$$
{\mathcal R}((\bm{\beta}_i)_{i=1}^n,\, \omega') \ =\ (b_s,\, \omega'_{s+1}) \ \neq\ 0
$$
(see \eqref{g9}). But this contradicts \eqref{e03}. Therefore, we conclude that
\eqref{g5} holds for every $1\, \leq\, i\, \leq\, n$. This completes the proof of the theorem.
\end{proof}

\section{Symplectic Structures}

\subsection{Induced symplectic form}

Let $V_1$ be a complex vector space, not necessarily finite dimensional. Let
$$A\,:\, V_1\otimes V_1\, \longrightarrow\, \mathbb{C}$$ be an alternating bilinear form. The
form gives
an element $$\omega \,\in\, \bigwedge\nolimits^2 V_1^*.$$ This element $\omega$
of $\bigwedge\nolimits^2 V_1^*$ induces a linear
map $$\widetilde{\omega}\,:\, V_1 \,\longrightarrow\, V_1^* .$$
Note that there are two possible choices of $\omega'$: any $v\, \in\, V_1$ is sent to the
map $w\, \longmapsto\, A(v,\, w)$ or to the map $w\, \longmapsto\, A(w,\, v)$. These two
homomorphisms differ only by a sign. We say that the form $A$
is {\em symplectic} if the homomorphism $\widetilde{\omega}$ is injective, in which case
$\widetilde{\omega}$ is also called symplectic.

Let $A$ be a symplectic structure on $V_1$.
Let $V_2$ be a linear subspace of $V_1$. We have the following sequence of linear maps:
\begin{equation}\label{eqn:ob1}
0\,\lra\, V_2 \,\lra\, V_1 \, \stackrel{\widetilde{\omega}}{\longrightarrow}\, V_1^* \,\lra
\, V_2^* \,\lra\, 0.
\end{equation}
Let $$f \, :\, V_2\,\longrightarrow \,V_2^*$$ be the composition maps in \eqref{eqn:ob1}; let
\begin{equation}\label{22}
K\, :=\, \operatorname{Ker}(f)\, \subset\, V_2
\end{equation}
be the kernel of $f$. Hence $f$ induces an injective homomorphism
\begin{equation}\label{24}
{\phi}\,:\, V_2/K\, \hookrightarrow\, V_2^* .
\end{equation}

We address the question that asks whether the homomorphism $\phi$
in \eqref{24} induces a symplectic form on $V_2/K$.

Consider the short exact sequence 
\begin{equation}\label{21}
0 \,\lra\, (V_2/K)^* \,\stackrel{\alpha}{\longrightarrow}\, V_2^* \,\stackrel{\alpha'}{\lra}\, K^* \,\lra\, 0,
\end{equation}
where $\alpha$ is the dual of the quotient map $V_2\, \longrightarrow\, V_2/K$ and the surjective homomorphism
$V_2^* \,\lra\, K^*$ is the dual of the inclusion map in \eqref{22}.
It is straightforward to see that we have an induced symplectic structure on $V_2/K$ if 
\begin{equation}\label{23}
\operatorname{Im}(\phi)\ \subseteq\ \operatorname{Im}(\alpha),
\end{equation}
where $\phi$ and $\alpha$ are the homomorphisms in \eqref{24} and \eqref{21} respectively.
Indeed, if \eqref{23} holds, then $\phi$ factors through a homomorphism
$$\widetilde{\phi}\,:\, V_2/K \,\hookrightarrow\, (V_2/K)^*.$$
In other words, $\widetilde{\phi}$ is uniquely determined by the following condition:
$$
\alpha \circ\widetilde{\phi}\ =\ \phi.
$$
This homomorphism $\widetilde{\phi}$ is anti--symmetric because $\widetilde{\omega}$ in \eqref{eqn:ob1}
is so.

To prove that \eqref{23} holds, take any $v\, \in\, V_2$; its image in $V_2/K$ will be denoted by
$\widehat{v}$. We have
$$
\phi(\widehat{v})(w) \,=\, A(v,\, w)
$$
for all $w\, \in\, V_2$. Restrict $\phi(\widehat{v})$ to $K$. For $w\,\in\, K$, we have
$$
\phi(\widehat{v})(w) \,=\, A(v,\, w)\,=\, - A(w,\, v)\,=\, - f(w)(v) \,=\, 0
$$
(see \eqref{22}). This implies that \eqref{23} holds.

Therefore, we have proved the following:

\begin{proposition}\label{prop1}
Let $A$ be a symplectic form on a vector space $V_1$. Let $V_2$ be a subspace, and let $K\,=\,
\operatorname{Ker}(f)$ be the kernel of the homomorphism $f$ (constructed as in \eqref{22}). Then
$A$ induces a symplectic structure $\overline{A}$ on the quotient vector space $V_2/K$.
\end{proposition}

Note that we can describe the above subspace $K$ as
$${}^LK\,\,=\,\, \{ v \,\in\, V_2\,\,\big\vert\,\, A(v,\,w)\,=\,0 \ \ \forall \ \, w\,\in\, V_2\}.$$
Since the form $A$ is bilinear and anti--symmetric, ${}^LK$ is a vector subspace of $V_2$ and it
coincides with the subspace
$${}^RK\,\,=\,\, \{ v \,\in\, V_2\,\,\big\vert\,\, A(w,\,v)\,=\,0 \ \ \forall \ \, w\,\in\, V_2\}.$$

\subsection{Canonical symplectic structures}\label{se2.2}

Let $M$ be a smooth manifold. Consider the cotangent bundle
\begin{equation}\label{2a}
p\, :\, T^*M\, \longrightarrow\, M.
\end{equation}
We recall that there is a natural $1$--form $\theta$ on $T^*M$ which is constructed as follows:
For any $x\, \in\, M$ and any $w\, \in\, (T_x M)^*$, we have $\theta (v)\,=\, w(dp(v))$ for all
$v\, \in\, T_w (T^*M)$, where $dp\,:\, T(T^*M) \, \longrightarrow\, TM$ is the differential of
the projection $p$ in \eqref{2a}. This form $\theta$ is known as the {\em Liouville one--form}. Then
\begin{equation}\label{e0}
\omega_{T^*M} \ = \ d\theta
\end{equation}
is a symplectic form on the total space of $T^*M$, i.e., it is a nondegenerate
closed $2$--form on $T^*M$.

If $M$ is a complex manifold, then $\theta$ and $\omega_{T^*M}$ are holomorphic forms.
If $M$ is a smooth variety, then both $\theta$ and $\omega_{T^*M}$ are algebraic forms.

\subsection{Liouville form on groups}\label{sec:Liouvilleforgroups}

Let $\mathscr{G}$ be a Fr\'echet Lie group not--necessarily finite dimensional. Its
Lie algebra will be denoted by $\mathfrak{s}$. Then
the cotangent bundle $T^*\mathscr{G}$ is trivial and it is identified with the vector
bundle $\mathscr{G}\times \mathfrak{s}^*\, \longrightarrow\, \mathscr{G}$.
The tangent bundle of $T^*\mathscr{G}$ has the following description
$$
T(T^*\mathscr{G})\ =\ 
(T^*\mathscr{G})\times (\mathfrak{s}\times \mathfrak{s}^*)\ =\
(\mathscr{G}\times \mathfrak{s}^*)\times (\mathfrak{s}\times \mathfrak{s}^*).$$

Take any $z \,=\, (g,\, w)\, \in\, \mathscr{G}\times \mathfrak{s}^*\,=\,
T^*\mathscr{G}$. Consider $T_z (T^*\mathscr{G})\,=\, \mathfrak{s}\times \mathfrak{s}^*$.
Consider the canonical $1$--form
$$\omega\ :\ T(T^*\mathscr{G})\ \longrightarrow\ {\mathbb C}$$
on $T^*\mathscr{G}$. Its restriction to $T_z (T^*\mathscr{G})\,=\,
\mathfrak{s}\times \mathfrak{s}^*$ has the following description:
$$
\omega(z) (u,\, v)\ =\ w(u) \ \in\ {\mathbb C},
$$
where $u\, \in \mathfrak{s}$ and $v\, \in\, \mathfrak{s}^*$.

Next we will describe the Liouville symplectic form $d\omega$ on $T^*\mathscr{G}$.

As before, take $z \,=\, (g,\, w)\, \in\, T^*\mathscr{G}$. For $k\,=\, 1,\, 2$,
take tangent vectors at $z$
$$
\gamma_k\ = \ (u_k,\, v_k)\ \in\ T_z (T^*\mathscr{G})\ =\
\mathfrak{s}\times \mathfrak{s}^*,
$$
so $u_k\, \in\, \mathfrak{s}$ and $v_k\, \in\, \mathfrak{s}^*$. Now we have
$$
d\omega (z)(\gamma_1,\, \gamma_2)\ =\ v_2(u_1) - v_1(u_2)\ \in\ {\mathbb C}.
$$

\subsection{Symplectic structures on the quotient}

Let $M$ be smooth manifold equipped with a symmetric bilinear form $\omega$ which is
nondegenerate. Here nondegeneracy means the following: Let
\begin{equation}\label{e7}
\omega'\, :\, TM \,\hookrightarrow\, T^*M
\end{equation}
be the homomorphism produced by this nondegenerate symmetric bilinear form $\omega$; it should be
clarified that nondegeneracy of $\omega$ means that the homomorphism $\omega'$ is fiber--wise injective.
Note that this condition implies that $\omega'$ is an isomorphism if $M$ is a finite dimensional manifold.

Assume that a Lie group $\mathscr{G}$ acts on $M$ such that the above
nondegenerate symmetric bilinear form $\omega$ on
$M$ is preserved by the action of $\mathscr{G}$. The Lie algebra of $\mathscr{G}$ will be
denoted by ${\mathfrak g}_0$. The action of $\mathscr{G}$ on $M$ produces a homomorphism
\begin{equation}\label{e6}
\phi\, :\, M\times {\mathfrak g}_0\, \longrightarrow\, TM
\end{equation}
from the trivial vector bundle $M\times {\mathfrak g}_0\, \longrightarrow\, M$ on $M$
with fiber ${\mathfrak g}_0$.

The action of $\mathscr{G}$ on $M$ induces an action of $\mathscr{G}$ on the tangent
bundle $TM$, and as well as on the cotangent bundle $T^*M$. Now 
consider the subsheaf $\mathcal{F}^{\mathscr{G}}$ of $TM$ given by the $\mathscr{G}$ orbits; its fiber
at a point $m\, \in\, M$ is $T_m(\mathscr{G}\cdot m)$. More precisely, $\mathcal{F}^{\mathscr{G}}$ is the
image of the homomorphism $\phi$ in \eqref{e6}, so we have
\begin{equation}\label{e20}
\mathcal{F}^{\mathscr{G}}\ =\ {\rm image}(\phi) \ \subset\ TM.
 \end{equation}
The annihilator of $\mathcal{F}^{\mathscr{G}}$ for the
nondegenerate symmetric bilinear form $\omega$ on $M$ will be denoted by
$(\mathcal{F}^{\mathscr{G}})^{\perp}$. (Since
$\omega$ is symmetric, $\mathcal{F}^{\mathscr{G}}$ does not depend on the two
choices available to define the annihilator.)

Let
\begin{equation}\label{e7a}
\mathcal{V}\ :=\ \omega'((\mathcal{F}^{\mathscr{G}})^{\perp})\ \subset \ T^*M
\end{equation}
be the image of the annihilator $(\mathcal{F}^{\mathscr{G}})^{\perp}$ under the
homomorphism $\omega'$ in \eqref{e7}, where $\mathcal{F}^{\mathscr{G}}$ is defined
in \eqref{e20}. Then $\mathcal{V}$ is a subsheaf of the cotangent 
bundle ${T^*}M$ of a manifold $M$. Note that $\mathcal{V}$ plays the role of the
annihilator of $\mathcal{F}^{\mathscr{G}}$ under the pairing $\omega$.

Consider the action of $\mathscr{G}$ on $T^*M$ induced by the action of $\mathscr{G}$ on $M$.
The following lemma shows that the subspace $\mathcal{V}$ in \eqref{e7a} is preserved by this
action of $\mathscr{G}$.

\begin{lemma}\label{lem1}
The action of $\mathscr{G}$ on $T^*M$ induced by the action of $\mathscr{G}$ on $M$,
preserves $\mathcal{V}$ constructed in \eqref{e7a}.
\end{lemma}

\begin{proof}
Consider the action of $\mathscr{G}$ on $TM$ induced by the action of $\mathscr{G}$ on $M$.
The homomorphism $\phi$ in \eqref{e6} is evidently $\mathscr{G}$--equivariant for the adjoint
action of $\mathscr{G}$ on its Lie algebra ${\mathfrak g}_0$ and the above actions of
$\mathscr{G}$ on $M$ and $TM$. This immediately implies that the
action of $\mathscr{G}$ on $TM$ preserves the subsheaf
$\mathcal{F}^{\mathscr{G}}\, \subset\, TM$ in \eqref{e20}.

Recall the given condition that the action of $\mathscr{G}$ on $M$ preserves the nondegenerate
symmetric bilinear form
$\omega$. Since the action of $\mathscr{G}$ on $TM$ preserves $\mathcal{F}^{\mathscr{G}}$,
this implies that the action of $\mathscr{G}$ on $TM$ also
preserves $(\mathcal{F}^{\mathscr{G}})^{\perp}$.

Consider the action of $\mathscr{G}$ on $T^*M$ induced by the action of $\mathscr{G}$ on $M$.
The homomorphism $\omega'$ in \eqref{e7} is $\mathscr{G}$--equivariant because the
nondegenerate symmetric bilinear form
$\omega$ is preserved by the action of $\mathscr{G}$ on $M$. Since
the action of $\mathscr{G}$ preserves $(\mathcal{F}^{\mathscr{G}})^{\perp}$, and
$\omega'$ is $\mathscr{G}$--equivariant, we conclude that the
action of $\mathscr{G}$ on $T^*M$ preserves $\omega'((\mathcal{F}^{\mathscr{G}})^{\perp})
\,=\, \mathcal{V}$.
\end{proof}

Consider the action of $\mathscr{G}$ on $\mathcal{V}$ obtained in Lemma \ref{lem1}.
Since the natural map
\begin{equation}\label{0p}
p\, :\, T^*M\, \longrightarrow\, M
\end{equation}
is $\mathscr{G}$--equivariant, the natural projection $\mathcal{V} \, \longrightarrow\, M$ is
also $\mathscr{G}$--equivariant. Consequently, we have
\begin{equation}\label{e8}
\mathcal{V}_{\mathscr{G}}\ :=\ \mathcal{V}/\mathscr{G} \ \longrightarrow\ M/\mathscr{G}.
\end{equation}

Now consider the Liouville one--form $\theta_M$ on $T^*M$ defined in Section \ref{se2.2}.
The following lemma says that the action of $\mathscr{G}$ on $T^*M$ preserves
$\theta_M$.

\begin{lemma}\label{lem2}
The action of $\mathscr{G}$ on $T^*M$, induced by the action of $\mathscr{G}$ on $M$,
preserves the Liouville one--form $\theta_M$ on $T^*M$ defined in Section \ref{se2.2}.
\end{lemma}

\begin{proof}
This a straight--forward computation using the definition of the Liouville one--form $\theta_M$.
\end{proof}

Restrict the Liouville one--form $\theta_M$ to $\mathcal{V}$ (defined in \eqref{e7a}). Let
\begin{equation}\label{e9}
\theta'_M \ \in \ H^0(\mathcal{V},\, T^*\mathcal{V})
\end{equation}
be the restriction of $\theta_M$ to $\mathcal{V}$.

Our aim is to find sufficient conditions ensuring the following: 
\begin{enumerate}
\item The $1$--form $\theta'_M$ on $\mathcal{V}$ (see \eqref{e9}) descends to a
$1$--form $\vartheta$ on the quotient $\mathcal{V}_{\mathscr{G}}$ in \eqref{e8}.

\item The exterior derivative $d\vartheta$ on $\mathcal{V}_{\mathscr{G}}$ is nondegenerate. 
\end{enumerate}

The following lemma ensures the first one of the above two requirements actually holds.

\begin{lemma}\label{prop:mainprop1}
The one--form $\theta'_M$ on $\mathcal{V}$ (defined in \eqref{e9}) descends to
the quotient $\mathcal{V}_{\mathscr{G}}$ in \eqref{e8}.
\end{lemma}

\begin{proof}
As before, ${\mathfrak g}_0$ denotes the Lie algebra of $\mathscr{G}$. The action of
$\mathscr{G}$ on $T^*M$, induced by the action of $\mathscr{G}$ on $M$, produces a homomorphism
\begin{equation}\label{e12}
\widetilde{\eta}\, :\, (T^*M)\times {\mathfrak g}_0\, \longrightarrow\, T(T^*M).
\end{equation}
Consider the projection $p$ in \eqref{0p}. Let
\begin{equation}\label{e13}
dp\, :\, T(T^*M) \, \longrightarrow\, TM
\end{equation}
be the differential of $p$. Let
\begin{equation}\label{e10}
\eta\,:=\, (dp)\circ \widetilde{\eta}\, :\, (T^*M)\times {\mathfrak g}_0\, \longrightarrow\, TM
\end{equation}
be the composition of maps, where $\widetilde{\eta}$ is constructed in \eqref{e12}.

{}From Lemma \ref{lem1} we know that the action of $\mathscr{G}$ on $T^*M$ preserves
$\mathcal{V}$. Also, Lemma \ref{lem2} says that $\mathscr{G}$ preserves
$\theta_M$. Consequently, it suffices to prove the following statement:

For any $x\, \in\, M$, $u\, \in\, \mathcal{V}_x\, \subset\, T^*_xM$ and $v\, \in\,
{\mathfrak g}_0$,
\begin{equation}\label{e11}
u (\eta (u,\, v))\ =\ 0,
\end{equation}
where $\eta$ is the map in \eqref{e10}. (Note that $\eta (u,\, v)\, \in\, T_xM$
and $u (\eta (u,\, v))\, \in\, {\mathbb C}$ because $u\, \in\, T^*_xM$.)

To prove \eqref{e11}, 
first note that the maps $\phi$ (in \eqref{e6}) and $\eta$ (in \eqref{e10})
are related as follows: For any $y\, \in\, M$, $u'\, \in\, T^*_yM$ and
$v'\, \in\, {\mathfrak g}_0$,
\begin{equation}\label{e14}
\eta(u',\, v')\ = \ \phi(p(u'),\, v'),
\end{equation}
where $p$ is the projection in \eqref{0p}. Clearly, we have $\phi(p(u'),\, v')\, \in\,
\mathcal{F}^{\mathscr{G}}_y$, where $\mathcal{F}^{\mathscr{G}}$ is defined in \eqref{e20}.
Therefore, for any
$$t\, \in\, ((\mathcal{F}^{\mathscr{G}})^{\perp})_y \, \subset\, T_yM,$$
we have $\omega(y)(t,\, \phi(p(u'),\, v'))\,=\, 0$, where $\omega$ is the
nondegenerate symmetric bilinear form on $M$. This implies that
$$
\theta_M(\omega'(t)) (\phi(p(u'),\, v'))\ =\ \omega'(t)(\phi(p(u'),\, v'))
\ =\ 0.
$$
So \eqref{e14} gives that $\theta_M(\omega'(t))(\eta(u',\, v'))\,=\, 0$.
But $\theta_M(\omega'(t))(\eta(u',\, v'))\,=\, \omega'(t)(\eta(u',\, v'))$,
and hence we conclude that $\omega'(t)(\eta(u',\, v')) \, =\, 0$. From this
it follows immediately that \eqref{e11} holds. This completes the proof of the lemma.
\end{proof}

Let
\begin{equation}\label{ev}
\vartheta \in H^0(\mathcal{V}_{\mathscr{G}}, T^*\mathcal{V}_{\mathscr{G}})
\end{equation}
denote the unique $1$--form on the quotient $\mathcal{V}_{\mathscr{G}}$ in \eqref{e8}
whose pullback to $\mathcal{V}$ is $\theta'_M$; the existence of $\vartheta$ is
ensured by Lemma \ref{prop:mainprop1}. Consequently, the pullback of the $2$--form
$d\vartheta$ to $\mathcal V$ coincides with $d\theta'_M$.

\begin{proposition}\label{prop:mainprop}
The exterior derivative $d\vartheta$ (see \eqref{ev} for $\vartheta$) is a
nondegenerate $2$--form on $\mathcal{V}_{\mathscr{G}}$. In particular the total
space of $\mathcal{V}_{\mathscr{G}} \,\longrightarrow\, M/\mathscr{G}$ inherits
a symplectic structure from the symplectic structure $d\theta_M$ on $T^*M$. 
\end{proposition}

\begin{proof}
In view of Proposition \ref{prop1}, the nondegeneracy of $d\vartheta$ is in fact a
consequence of the nondegeneracy of $d\theta_M$. To see this, recall that $\mathcal{V}$
is preserved by the action of $\mathscr{G}$ on $T^*M$ (see Lemma \ref{lem1}). Let
\begin{equation}\label{h0}
\varphi\, :\, \mathcal{V}\times {\mathfrak g}_0\, \longrightarrow\, T\mathcal{V}
\end{equation}
be the homomorphism given by the action of $\mathscr{G}$ on $\mathcal{V}$, where ${\mathfrak g}_0$,
as before, is the Lie algebra of $\mathscr{G}$.

Take a point
$$
q\, \in\, \omega'((\mathcal{F}^{\mathscr{G}})^{\perp})\,=\, \mathcal{V}
$$
(see \eqref{e7a}). The image of $q$ in $\mathcal{V}/\mathscr{G}\,=\, \mathcal{V}_{\mathscr{G}}$
will be denoted by $\overline{q}$. It can be shown that the tangent space $T_{\overline{q}}
\mathcal{V}_{\mathscr{G}}$ has a natural identification
\begin{equation}\label{h1}
T_{\overline{q}} \mathcal{V}_{\mathscr{G}}\ = \ (T_q\mathcal{V})/(\varphi(q,\, {\mathfrak g}_0)),
\end{equation}
where $\varphi$ is the map in \eqref{h0}. Indeed, \eqref{h1} is an immediate consequence
of the properties of a quotient space.

The Liouville symplectic form $d\theta_M$ on $T^*M$ (see \eqref{e0}) produces a
a fiber--wise injective homomorphism
\begin{equation}\label{h2}
\beta\, :\, T(T^*M)\, \hookrightarrow\, T^*(T^*M)
\end{equation}
(see \eqref{e7}).

Using the inclusion map $\mathcal{V}\, \hookrightarrow \, T^*M$ in \eqref{e7a}, we have
the maps 
\begin{equation}\label{eqn:fund2}
T_q \mathcal{V}\, \hookrightarrow\, T_q(T^*M)\, \xrightarrow{\,\,\,\beta(q)\,\,\,}\,
T^*_q(T^*M) \, \twoheadrightarrow\, T^*_q \mathcal{V},
\end{equation}
where $\beta$ is the homomorphism in \eqref{h2}. Let
$$
\rho\, :\, T_q \mathcal{V}\, \longrightarrow\, T^*_q \mathcal{V}
$$
be the composition of maps in \eqref{eqn:fund2}. It is straightforward to check that
$\text{kernel}(\rho)$ coincides with $\varphi(q,\, {\mathfrak g}_0)$. From \eqref{h1}
we know that the quotient $(T_q \mathcal{V})/(\varphi(q,\, {\mathfrak g}_0))$
coincides with $T_{\overline{q}} \mathcal{V}_{\mathscr{G}}$.
Consequently, using Proposition \ref{prop1} we conclude that
$d\vartheta$ is nondegenerate. Hence $d\vartheta$ is a symplectic form on $\mathcal{V}_{\mathscr{G}}$.
\end{proof}

\section{Symplectic structure on moduli of Higgs bundles with framings}

Let $G$ be a semisimple and simply connected affine algebraic group defined over $\mathbb{C}$. 
The Lie algebra of $G$ will be denoted by $\mathfrak{g}$. Let $X$ be an irreducible
smooth complex projective curve equipped
with a marked point $p \,\in\, X$. We first recall the uniformization of the stack of 
principal $G$--bundles on $X$.

\subsection{Uniformization of principal $G$--bundles}

Let $t$ be a formal parameter considered as a holomorphic coordinate at the point $p\,\in\, X$. By 
Harder's theorem, any principal $G$--bundle on $X \backslash \{p\}$ is trivial \cite{Ha}.
Hence by the uniformization theorem \cite{Faltings:94, KNR:94, BeauvilleLaszlo:94}, the following
is obtained:

The $\mathbb{C}$--valued points of the moduli stack of principal $G$--bundles on
the projective curve $X$ can be described as the double quotient
\begin{equation}
\label{eqn:uniform}
\operatorname{Bun}_G(X)\ =\ G[[t]]\backslash G((t)) / G(X\backslash \{p\}),
\end{equation}
where $G(X\backslash \{p\})$ is the space of algebraic maps from $X\backslash \{p\}$ to the 
group $G$. For notational conveniences, the loop group $G((t))$ is also denoted by 
$LG$, while the group $G[[t]]$ of positive loops is denoted by $L^+G$ and the
subgroup $G(X\backslash \{p\})$ is denoted by $L_XG$. So, \eqref{eqn:uniform} can be re--written as
$$
\operatorname{Bun}_G(X)\ =\ L^+G\backslash LG / L_XG.
$$

We can consider the elements of the loop group $LG$ as the $\mathbb{C}$--valued points of the 
moduli stack of principal $G$--bundles on $X$ equipped with a chosen trivialization on the complement
$X\backslash \{p\}$ and a chosen trivialization on the formal disc $D_{p}$ around the point 
$\{p\}$. Similarly the ind--variety $L^+G \backslash LG$ parametrizes principal $G$--bundle on 
$X$ equipped with a chosen trivialization on the complement $X\backslash \{p\}$.

Now consider the cotangent bundle $T^*LG$ of the loop group $LG$.
The cotangent bundle $T^*LG$ is trivial, in fact, it is identified with the
trivial vector bundle
\begin{equation}\label{e15}
S\ :=\ LG\times \left( \mathfrak{g}\otimes K_{D_{p}}((t))\right);
\end{equation}
here we have 
identified $\mathfrak{g}$ with its dual $\mathfrak{g}^*$
using the normalized Cartan--Killing form $(-,\, -)$ on $\mathfrak g$. The term $K_{D_{p}}$ in
\eqref{e15}
is the canonical bundle of a formal neighborhood $D_p$ of $p\, \in\, X$. The group $LG$ acts on the cotangent 
bundle $S$ of $LG$.

The loop group $LG$ carries a natural nondegenerate symmetric bilinear form. It is defined as follows:
\begin{equation}\label{e16}
\langle A\otimes t^m,\, B\otimes t^n\rangle \ =\ \delta_{n+m,0}\left(A,\, B\right),
\end{equation}
where $A,\, B\, \in\, {\mathfrak g}$, and $\left(- ,\, -\right)$ is the normalized Cartan--Killing
form on $\mathfrak{g}$, while
$\delta_{i,j}\,=\, 0$ if $i\, \not=\, j$ and it is $1$ if $i\,=\, j$ (see \cite{PS}).
Note that \eqref{e16} defines a bilinear form $\langle -, \, -\rangle$ on the vector space
${\mathfrak g}((t))$ using bilinearity. The bilinear form on ${\mathfrak g}((t))$ is
evidently symmetric. It is also nondegenerate, meaning the homomorphism
\begin{equation}\label{j2}
{\mathfrak g}((t)) \ \longrightarrow\ {\mathfrak g}((t))^*
\end{equation}
given by the pairing is injective. To see the injectivity of the homomorphism in \eqref{j2},
for any $A\otimes t^m$, where $A\, \in\, {\mathfrak g}$, take any $B\, \in\, {\mathfrak g}$
such that $(A,\, B)\, \not=\, 0$. Now, clearly we have
$\langle A\otimes t^m,\, B\otimes t^{-m}\rangle \, \not=\, 0$, and hence the bilinear
form on ${\mathfrak g}((t))$ is nondegenerate.

So \eqref{e16} defines a nondegenerate symmetric bilinear form on the tangent space of $LG$ at the identity
element of $LG$. Now extend this to a nondegenerate symmetric bilinear form on $LG$ using left--translations.

This nondegenerate symmetric bilinear form on $LG$ given by \eqref{e16} will be denoted by
$\omega$. As in \eqref{e7}, let
\begin{equation}\label{e21}
\omega'\ :\ T(LG) \ \hookrightarrow\ T^*(LG)
\end{equation}
be the homomorphism given by $\omega$; note that $\omega'$ is fiber--wise injective because
\begin{itemize}
\item $\omega'$ injective on the fiber $T_e(LG)$ over the identity element of $LG$, and 

\item $\omega'$ is $LG$--equivariant.
\end{itemize}

{}From the point of view of Proposition \ref{prop:mainprop}, the role of $M$ in 
that proposition will be played by $LG$ while the role of
$\mathcal G$ will be played by $L_XG$. In the rest of this section we will consider various 
subbundles of $T^*LG$ to which Proposition \ref{prop:mainprop} can be applied.

\subsection{Principal bundles with meromorphic Higgs fields}\label{sec:arbitraryframing}

Let $E_G$ be a principal $G$--bundle on $X$. The adjoint bundle for $E_G$ will be denoted by
$\text{ad}(E_G)$; we recall that $\text{ad}(E_G)$ is
the Lie algebra bundle on $X$ associated to the principal $G$--bundle $E_G$ for the
adjoint action of $G$ on its Lie algebra $\mathfrak g$. A \textit{meromorphic Higgs field} on
$E_G$ is a meromorphic section of the direct limit
\begin{equation}\label{vp}
\varphi\, \in\, H^0(X,\, \text{ad}(E_G)\otimes K_X(\ast p)) \,=\,
\lim_{i\rightarrow \infty} H^0(X,\, \text{ad}(E_G)\otimes K_X{\mathcal O}_X(i p)),
\end{equation}
where $K_X$, as before, is the canonical line bundle of $X$;
the above direct limit is constructed using the natural inclusion maps
$\text{ad}(E_G)\otimes{\mathcal O}_X(j p)\, \hookrightarrow\, \text{ad}(E_G)\otimes{\mathcal O}_X((j+k) p)$
where $k\, \geq\, 0$.

A \textit{meromorphic principal Higgs} $G$--\textit{bundle} is a principal $G$--bundle
on $X$ equipped with a meromorphic Higgs field.

We will construct a subbundle
\begin{equation}\label{e17}
\mathcal{W} \ \subset \ T^*LG
\end{equation}
of the cotangent bundle $T^*LG$. To describe the fibers of $\mathcal{W}$ point--wise, take any 
point $\alpha\,\in\, LG$. The point $\alpha$ gives a principal $G$--bundle $E_G$ on $X$ with a 
given trivialization of $E_G$ on $X\backslash \{p\}$ and a given trivialization of $E_G$ on 
formal completion $D_{p}$. Now consider the adjoint vector bundle $\ad(E_G)$ on $X$. Note that the
elements of $\mathfrak{g}\otimes K_{D_{p}}((t))$ can be considered as sections of $(\ad(E_G)\otimes 
K_X)_{D_{p}}$ with pole of arbitrary order at $p$. The fiber $\mathcal{W}_\alpha$ consists of all 
elements of $\mathfrak{g}\otimes K_{D_{p}}((t))$ that extend to a section of $\ad(E_G)\otimes 
K_X$ over the complement $X\setminus \{p\}$. So the only point of $X$ where such a section can 
have pole is $p$; the order, at $p$, of the pole of this meromorphic section can be arbitrary.

Recall that $L_XG\,\hookrightarrow\, LG$, and consequently $L_XG$ acts on $LG$ via left--translations.
This action of $L_XG$ on $LG$ induces actions of $L_XG$ on both $TLG$ and $T^*LG$.

\begin{lemma}\label{lem2a}
The action of $L_XG$ on $T^*LG$, induced by the action of $L_XG$ on $LG$, preserve the
subbundle $\mathcal{W}$ in \eqref{e17}.
\end{lemma}

\begin{proof}
We recall that an element $\alpha$ of $LG$ gives a principal $G$--bundle $E_G$ on $X$ with given 
trivializations of $E_G$ over $X\backslash \{p\}$ and the formal
completion $D_{p}$. It is easy to see that we have
$\alpha \,\in\, L_XG\, \subset\, LG$ if and only if the trivialization of $E_G$ over $D_{p}$ extends to a
trivialization of $E_G$ over entire $X$. So we get two trivializations
of $E_G\big\vert_{X\backslash \{p\}}$: One given directly by $\alpha$ and the other obtained by
extending, to entire $X$, the trivialization of $E_G\big\vert_{D_p}$ given by $\alpha$.
These two trivializations of $E_G\big\vert_{X\backslash \{p\}}$ differ by an automorphism 
of $E_G\big\vert_{X\backslash \{p\}}$.

Now take any $\beta\, \in\, LG$. it gives principal $G$--bundle $F_G$ on $X$ with given
trivializations of $F_G$ over $X\backslash \{p\}$ and $D_{p}$. As before, take any
$\alpha \,\in\, L_XG$. Then the element $\beta\alpha\, \in\, LG$ gives the same principal
$G$--bundle $F_G$ on $X$ (the principal $G$--bundle given by $\beta$), and the trivialization
of $F_G$ over $D_{p}$ for $\beta\alpha$ remains unchanged (it coincides with the one given by
$\beta$). But the trivialization of $F_G$ over $X\backslash \{p\}$ for
$\beta\alpha$ changes by the automorphism of the trivial principal $G$--bundle given by $\alpha$.

Take $\beta\, \in\, LG$ as above. Then the fiber $\mathcal{W}_\beta$ of $\mathcal{W}$ (see 
\eqref{e17}) over $\beta$ is canonically identified with the space of all meromorphic Higgs 
fields on the principal $G$--bundle $F_G$ given by $\beta$.

{}From the above descriptions of $\mathcal{W}_\beta$, and the action of $L_XG$ on $LG$, it follows
immediately that the action of $L_XG$ on $T^*LG$, induced by the action of $L_XG$ on $LG$,
preserve the subbundle $\mathcal{W}$ in \eqref{e17}.
\end{proof}

As mentioned before, from the point of view of Proposition \ref{prop:mainprop} the role of
$\mathscr{G}$ will be played 
by $L_XG$. Consider the quotient $\mathcal{W}/{L_XG}$ which is a subbundle of 
the quotient $\left(T^*LG\right)/{L_XG}$. Note that
$$
\mathcal{W}/{L_XG} \ \subset\ \left(T^*LG\right)/{L_XG}
$$
are vector bundles on the space $LG/L_XG$. It may be mentioned that $\mathcal{W}/{L_XG}$ is
precisely the moduli stack of principal $G$--bundles $E_G$ on $X$, equipped with
\begin{itemize}
\item an arbitrary order framing of $E_G$ at $p$ (meaning a trivialization of $E_G$ over
$D_{p}$), and

\item a meromorphic Higgs field on $E_G$.
\end{itemize}

We are in a position to prove the following theorem.
 
\begin{theorem}\label{thm:main1}
Consider the moduli stack $\mathcal{W}/{L_XG}$ parametrizing principal $G$--bundles on
$X$ equipped with an arbitrary order framing at $p$ 
and a meromorphic Higgs field. It inherits a canonical symplectic 
structure constructed using the Liouville symplectic structure on $T^*LG$.
\end{theorem}

\begin{proof}
We need to put ourselves in the set--up of Proposition \ref{prop:mainprop}
in order to apply it. As in \eqref{e20}, construct the subsheaf
\begin{equation}\label{e23}
\mathcal{F}^{{L_XG}}\ \subset \ T(LG)
\end{equation}
using the left--translation action of the subgroup $L_XG$ on $LG$. Let $$(\mathcal{F}^{{L_XG}})^\perp
\, \subset \, T(LG)$$ be the annihilator of $\mathcal{F}^{{L_XG}}$ in
\eqref{e23} for the nondegenerate symmetric bilinear form 
$\omega$ on $LG$ (see \eqref{e21} and \eqref{e16}). So we have
\begin{equation}\label{e22}
{\mathcal V} \ :=\ \omega'(\mathcal{F}^{{L_XG}})\ \subset \ T^*(LG),
\end{equation}
where $\omega'$ is the homomorphism in \eqref{e21}.

To prove the theorem it is enough to show that the vector subbundle
$\mathcal{V}\, \subset\, T^*LG$ (see \eqref{e22}) coincides with the subbundle
$\mathcal{W} \ \subset \ T^*LG$ in \eqref{e17}.

As before, $t$ is a formal parameter considered as a holomorphic coordinate
at the point $p\,\in\, X$.
Take any element $\alpha \,\in\, LG$. The fiber $T_\alpha (LG)$ of the tangent bundle
$T(LG)$ over the point $\alpha$ is identified with ${\mathfrak g}((t))$. Indeed, this
follows immediately from the fact that the Lie algebra of the loop group $LG$ is
${\mathfrak g}((t))$.

Let $E_G$ denote the principal $G$--bundle on $X$ given by $\alpha\, \in\, LG$.
Recall that $\alpha$ gives a trivialization of the principal $G$--bundle $E_G$ over
the formal disc $D_p$ around the point $p\, \in\, X$. This trivialization of
the principal $G$--bundle $E_G$ over $D_p$ produces a trivialization of
the adjoint vector bundle $\ad(E_G)$ over $D_p$. More precisely, the restriction
of $\ad(E_G)$ to $D_p$ is identified with the trivial Lie algebra bundle $D_p
\times{\mathfrak g}\,\longrightarrow\,D_p$ over $D_p$ with fiber $\mathfrak g$. 

Using this trivialization of $\ad(E_G)$ over $D_p$,
for any $\sigma \, \in\, H^0(X,\, \ad(E_G)\otimes\mathcal{O}_X(j p))$, where $j\, \geq\,1$,
by taking the Laurent expansion of $\sigma$ around $p$ we get an element of
${\mathfrak g}((t))\,=\,T_\alpha (LG)$. Therefore, we have an injective homomorphism
\begin{equation}\label{e24}
H^0(X,\, \ad(E_G)\otimes\mathcal{O}_X(\ast p))\,:=\,
\lim_{j\rightarrow \infty} H^0(X,\, \ad(E_G)\otimes\mathcal{O}_X(j p))
\, \hookrightarrow\, {\mathfrak g}((t))\,=\,T_\alpha (LG).
\end{equation}
It should be mentioned that the above homomorphism $H^0(X,\, \ad(E_G)\otimes\mathcal{O}_X(\ast p))\,
\hookrightarrow\, {\mathfrak g}((t))$ is Lie algebra structure preserving.

The fiber $(\mathcal{F}^{L_XG})_\alpha$ of $\mathcal{F}^{L_XG}$ (see \eqref{e23})
is the Lie algebra $H^0(X,\, \ad(E_G)\otimes\mathcal{O}_X(\ast p))$; it should be
clarified that $H^0(X,\, \ad(E_G)\otimes\mathcal{O}_X(\ast p))$ is considered as a
subspace of $T_\alpha (LG) \,=\, {\mathfrak g}((t))$ using the homomorphism in \eqref{e24}.
That $(\mathcal{F}^{L_XG})_\alpha \,=\, H^0(X,\, \ad(E_G)\otimes\mathcal{O}_X(\ast p))$
is a straightforward consequence of the fact that the Lie algebra of $L_XG$ is
$$\mathfrak{g}\otimes H^0(X, \mathcal{O}_X(\ast p))\ =\ \mathfrak{g}\otimes (\lim_{j\rightarrow \infty} 
H^0(X,\, \mathcal{O}_X(j p))).$$

Now from Theorem \ref{thm:residueann} it follows immediately that $$(H^0(X, \ad(E_G)(\ast p)))^{\perp}
\ =\ (H^0(X, \ad(E_G)\otimes K_X(\ast p))).$$ This proves the assertion that
the vector subbundle
$\mathcal{V}\, \subset\, T^*LG$ (see \eqref{e22}) actually coincides with the subbundle
$\mathcal{W} \ \subset \ T^*LG$ in \eqref{e17}. As observed before, this statement completes
the proof of the theorem.
\end{proof}

\subsection{Higgs fields with pole of bounded order}\label{se-hf}

Fix a positive integer $k$. We will consider Higgs fields with pole at $p$ of order
at most $k$. Let $E_G$ be a principal $G$--bundle on $X$. A \textit{Higgs field on} $E_G$
\textit{with a pole of order $k$} is an element of $H^0(X,\, \text{ad}(E_G)\otimes K_X
\otimes {\mathcal O}_X(kp))$, where $\text{ad}(E_G)$ is the adjoint bundle for $E_G$.

Before we proceed further we recall the notion of framing of a principal $G$ bundle at a point $p\in X$. 
\begin{definition}\label{def:kframing}
A {\em $k$--th order framing } of a principal $G$ bundle $E_G$ at a point $p$ is a choice of a trivialization of ${E_{G}}_{|(k+1)p}$.
\end{definition}
We now consider the moduli stack of principal $G$--bundles on $X$ with a
$k$--th order framing and a Higgs field with pole of order 
$k$. Following the strategy of Section \ref{sec:arbitraryframing}, a symplectic
structure on it will be constructed.

Consider $\mathcal W$ constructed in \eqref{e17}.
Let $\mathcal{W}^{k}$ be the subbundle of $\mathcal{W}$
which is described fiber--wise as follows: Take any $\alpha\,\in\, LG$. Denote by $E_G$
the principal $G$--bundle on $X$ given by $\alpha$.
Recall that the fiber $\mathcal{W}_{\alpha}$ of $\mathcal{W}$ over $\alpha$
consists of all elements of $\mathfrak{g}\otimes K_{D_{p}}((t))$ that 
extend to a meromorphic section of $\ad(E_G)\otimes K_X$ with allowed pole only at
the fixed point $p$ (so they are holomorphic on the complement $X\setminus \{p\}$).
Similarly define $\mathcal{W}^{k}_{\alpha}$ be the subspace of 
$\mathcal{W}_{\alpha}$ whose elements have pole of order at most $k$ at $p$.

Let $G^{k}[[t]]$ be the subgroup of the group of positive loops $G[[t]]$ consisting of 
elements of the form $e+ \sum_{j=0}^{\infty}g_jt^{k+j}$, where $e$ is the identity
element of $G$. 
Note that the group $G^{k}[[t]]\times L_X G$ acts on the loop group $LG$ and hence it also acts
on $T^*LG$.

We have the following analog of Lemma \ref{lem2a}.

\begin{lemma}\label{lem4}
The action of the group $G^{k}[[t]]\times L_XG$ on $T^*LG$ preserves the
above subbundle $\mathcal{W}^{k}$, thus producing a vector bundle
$\mathcal{W}^{k}_{G^{k}[[t]]\times L_X G}$ on $G^{k}[[t]]\backslash LG/L_XG$.
\end{lemma}

\begin{proof}
The Lie algebra of $G^{k}[[t]]$ is the subspace
$$
t^k\cdot {\mathfrak g}[[t]] \ \subset\ \ \in\ {\mathfrak g}[[t]].
$$
Consider the residue pairing ${\mathcal R}$ in \eqref{e3}. It can be shown that
the annihilator of the above subspace $t^k\cdot {\mathfrak g}[[t]]$ for $\mathcal R$ is
$t^{-k} {\mathfrak g}[[t]]dt$. Indeed, clearly,
$$
{\mathcal R}(t^k\cdot {\mathfrak g}[[t]],\, t^{-k} {\mathfrak g}[[t]]dt)\ =\ 0
$$
because for any $v\, \in\, t^k\cdot {\mathfrak g}[[t]]$ and $w\, \in\,
t^{-k} {\mathfrak g}[[t]]dt$, their tensor product $v\otimes w$ does not have a pole.
So the annihilator of $t^k\cdot {\mathfrak g}[[t]]$ contains $t^{-k} {\mathfrak g}[[t]]dt$.

To prove that the annihilator of $t^k\cdot {\mathfrak g}[[t]]$ is contained
in $t^{-k} {\mathfrak g}[[t]]dt$, take any
$$
w\ \in\ {\mathfrak g}((t))dt\, \setminus \, t^{-k} {\mathfrak g}[[t]]dt
$$
lying in the complement. Let $t^{-k-\ell}\beta$ be the first nonzero term of $w$; so
$\ell \, \geq\, 1$ and $\beta\, \in\, {\mathfrak g}$. Take any $\beta'\, \in\,
{\mathfrak g}$ such that
$$
(\beta\, ,\,\beta') \ \not= \ 0,
$$
where $(-\, ,\,-)$ is the normalized Cartan--Killing form on $\mathfrak g$. Now note
that
$$
{\mathcal R}(t^{k+\ell-1}\beta',\, w)\ =\ (\beta\, ,\,\beta') \ \not= \ 0.
$$
Consequently, the annihilator of $t^k\cdot {\mathfrak g}[[t]]$ is contained
in $t^{-k} {\mathfrak g}[[t]]dt$.

In view of the above observation, the lemma follows by using the argument in
the proof of Lemma \ref{lem2a}.
\end{proof}

Lemma \ref{lem4} allows us to apply Proposition \ref{prop:mainprop}, and we get 
the following theorem.

\begin{theorem}\label{th-j-f}
Consider the moduli stack parametrizing the principal $G$--bundles $E_G$ on $X$ with $k$--th order 
framing of $E_G$ at $p$ and a meromorphic Higgs field on $E_G$ with a pole of order at most
$k$ at $p$. This moduli stack has
a canonical symplectic structure coming from the Liouville symplectic structure on $T^*LG$.
\end{theorem}

The proof is similar to the proof of Theorem \ref{thm:main1} and we omit the details. 

\section{Principal bundles with framings and meromorphic connections}

\subsection{Meromorphic connections}

As before, $X$ is an irreducible smooth complex projective curve. Fix a
marked point $p \,\in\, X$. Let $\varpi\, :\, E_G\, \longrightarrow\, X$ be a
principal $G$--bundle on $X$.
The Atiyah bundle $\text{At}(E_G)$ of $E_G$ is the quotient $(TE_G)/G \,
\longrightarrow\, E_G/G \,=\, X$, which a vector bundle on $X$. Let
\begin{equation}\label{at}
0\, \longrightarrow\, \text{ad}(E_G) \, \longrightarrow\, \text{At}(E_G)
\, \xrightarrow{\,\,\, d\varpi\,\,\,}\, TX \, \longrightarrow\, 0
\end{equation}
be the Atiyah exact sequence on $X$, where $d\varpi$ is the differential of the
above projection $\varpi$ (see \cite{At}); note that $\text{ad}(E_G)\,=\, T_\varpi/G$,
where $T_\varpi\, \subset\, TE_G$ is the relative tangent bundle for the projection
$\varpi$.

We recall from \cite{At} that an algebraic connection on $E_G$ is a homomorphism
$$
{\mathcal D}\, :\, TX \, \longrightarrow\, \text{At}(E_G)
$$
such that $(d\varpi)\circ{\mathcal D}\,=\, {\rm Id}_{TX}$, where $d\varpi$ is the
homomorphism in \eqref{at}. A \textit{meromorphic connection} on $E_G$ is a homomorphism
$$
{\mathcal D}\, :\, (TX)\otimes {\mathcal O}_X(-np) \, \longrightarrow\, \text{At}(E_G),
$$
where $n$ is some nonnegative integer, such that $(d\varpi)\circ{\mathcal D}\,=\,
{\rm Id}_{(TX)\otimes {\mathcal O}_X(-np)}$; note that we have
$$
(TX)\otimes {\mathcal O}_X(-np)\, \subset\, TX
$$
because it is assumed that $n\, \geq\, 0$, so the composition $(d\varpi)\circ{\mathcal D}$ makes
sense.

If ${\mathcal D}\, :\, (TX)\otimes {\mathcal O}_X(-np) \, \longrightarrow\, \text{At}(E_G)$
is a homomorphism such that $$(d\varpi)\circ{\mathcal D}\,=\,
{\rm Id}_{(TX)\otimes {\mathcal O}_X(-np)},$$ then $\mathcal D$ will be a called a meromorphic connection
on $E_G$ with a pole of order at most $k$.

For the trivial principal $G$--bundle $E^0_G\,=\, M\times G$ on any smooth complex
variety $M$, we have $\text{At}(E^0_G)\,=\, \text{ad}(E^0_G)\oplus TM$. The algebraic
connection on $E^0_G$ defined by the natural inclusion map
$$
TM \, \hookrightarrow\, \text{ad}(E_G)\oplus TM \,=\, \text{At}(E^0_G)
$$
is called the trivial connection on $E^0_G$.

Every principal $G$--bundle $F_G$ on $X$ admits a meromorphic connection. Indeed, this
follows immediately from the fact that the restriction of $F_G$ to the complement
$X\setminus\{p\}$ is trivial; the trivial connection on $F_G\big\vert_{X\setminus\{p\}}$
is a meromorphic connection on $F_G$. The space of all meromorphic connections on
$F_G$ is an affine space for the vector space $H^0(X,\, \text{ad}(F_G)\otimes K_X
\otimes {\mathcal O}_X(\ast p))$. Recall that $H^0(X,\, \text{ad}(F_G)\otimes K_X
\otimes {\mathcal O}_X(\ast p))$ is the space of all meromorphic Higgs fields on the
principal $G$--bundle $F_G$.

Take any element $\alpha\, \in\, LG$. Recall that $\alpha$ gives a principal $G$--bundle
$E_G$ on $X$, and
\begin{itemize}
\item a trivialization of $E_G$ over the complement $X\setminus \{p\}$, and

\item a trivialization of $E_G$ over the formal disc $D_p$ around the point $p\, \in\, X$.
\end{itemize}
Consider the trivial connection on $E_G\big\vert_{X\setminus \{p\}}$ given by the
above trivialization of $E_G$ over $X\setminus \{p\}$. It evidently defines a meromorphic
connection on $E_G$. So for each $\alpha\, \in\, LG$, the corresponding
principal $G$--bundle $E_G$ on $X$ is equipped with a meromorphic connection given by $\alpha$.
This meromorphic connection on the principal $G$--bundle $E_G$ on $X$ given by $\alpha$
will be denoted by
\begin{equation}\label{l1}
{\mathcal D}_\alpha.
\end{equation}

\begin{lemma}\label{lem5}
Take any $\alpha\, \in\, LG$, and let $E_G$ be the principal $G$--bundle on $X$ corresponding
to $\alpha$. Then the space of all meromorphic connections on $E_G$ is canonically identified
with the space of all meromorphic Higgs fields on $E_G$.
\end{lemma}

\begin{proof}
Take any meromorphic Higgs field
$$
\varphi\, \in\, H^0(X,\, \text{ad}(E_G)\otimes K_X(\ast p)) \,=\,
\lim_{i\rightarrow \infty} H^0(X,\, \text{ad}(E_G)\otimes K_X{\mathcal O}_X(i p))
$$
on $E_G$ (see \eqref{vp}). Consider ${\mathcal D}_\alpha +\varphi$, where ${\mathcal D}_\alpha$
is the meromorphic connection on $E_G$ given by $\alpha$ (see \eqref{l1}). It is evident
that ${\mathcal D}_\alpha +\varphi$ is a meromorphic connection on the principal $G$--bundle
$E_G$. Conversely, if ${\mathcal D}$ is a meromorphic connection on the principal $G$--bundle
$E_G$, then ${\mathcal D}-{\mathcal D}_\alpha$ is a meromorphic Higgs field on $E_G$.
\end{proof}

{}From Lemma \ref{lem5} it follows immediately that for any element
$\alpha\, \in\, LG$, the fiber ${\mathcal W}_{\alpha}$ of ${\mathcal W}$ (see \eqref{e17})
over the point $\alpha$ is identified with the space of all meromorphic connections on
the principal $G$--bundle $E_G$ given by $\alpha$.

Let ${\mathcal U}$ denote the space of all pairs of the form $(\alpha,\, {\mathcal D})$, where
$\alpha\, \in\, LG$ and $\mathcal D$ is a meromorphic connection on the principal $G$--bundle
$E_G$ on $X$ given by $\alpha$. We have the natural projection
\begin{equation}\label{e31}
\Phi\,:\, {\mathcal U}\, \longrightarrow\, LG
\end{equation}
that sends any $(\alpha,\, {\mathcal D})\, \in\, {\mathcal U}$ to $\alpha$. So the
fiber ${\mathcal U}_\alpha\,=\, \Phi^{-1}(\alpha)$, where $\alpha\, \in\, LG$, is the
space of meromorphic connection on the principal $G$--bundle on $X$ given by $\alpha$.

Lemma \ref{lem5} has the following corollary:

\begin{corollary}\label{cor1}
The fiber bundle ${\mathcal U}\,\longrightarrow\, LG$ in \eqref{e31} is canonically
identified with the fiber bundle $\mathcal W$ in \eqref{e17}. In particular,
$\mathcal U$ is sub--vector bundle of the cotangent bundle $T^*LG$ (because $\mathcal W$ is
so).
\end{corollary}

We recall that the action of $L_XG$ on $T^*LG$, induced by the action of $L_XG$ on
$LG$, preserves the subbundle $\mathcal{W}$ (see Lemma \ref{lem2a}). Using the identification
of $\mathcal U$ with $\mathcal W$ given by Corollary \ref{cor1}, the action of $L_XG$
on $\mathcal W$ produces an action of $L_XG$ on $\mathcal U$. Clearly, the
map $\Phi$ in \eqref{e31} is equivariant for the actions of $L_XG$ on $\mathcal U$
and $LG$. However, $\mathcal U$ has a \textit{different} action of $L_XG$ which also
satisfies the condition that the map $\Phi$ is equivariant. This action of $L_XG$ on
$\mathcal U$ is actually constructed using a different action of $L_XG$ on $T^*LG$.
We will now describe the new action of $L_XG$ on $T^*LG$.

Take any $(\alpha,\, \theta)\in LG\times (K_{D_p}\otimes
\mathfrak{g}((\xi)))$ and $g\,\in\, L_XG$. Then define the action
\begin{equation}\label{e32}
(\alpha,\, \theta)\cdot g \ =\ (\alpha g, \, g^{-1}{\mathcal D}_\alpha (g) +
\operatorname{Ad}(g)(\theta)),
\end{equation}
where ${\mathcal D}_\alpha$ is the meromorphic connection in \eqref{l1}. It is straight--forward to
check that \eqref{e32} defines an action of $L_XG$ on $T^*LG$. The natural projection
$T^*LG\, \longrightarrow\, LG$ remains $L_XG$--equivariant for this new action of
$L_XG$.

\begin{lemma}\label{lem6}
The action of $L_XG$ on $T^*LG$ in \eqref{e32} preserves the subbundle ${\mathcal W}
\, \subset\, T^*LG$ in \eqref{e17}. 
\end{lemma} 

\begin{proof}
Let $E_G$ denote the principal $G$--bundle on $X$ given by $\alpha$. In \eqref{e32}, assume
that $\theta$ extends to a section of ${\rm ad}(E_G)\otimes K_X$ over the entire complement
$X\setminus\{p\}$. Since $g\, \in\, L_XG$, this implies that $g^{-1}{\mathcal D}_\alpha (g)$
is defined on entire $X\setminus\{p\}$; recall that ${\mathcal D}_\alpha$ in \eqref{l1}
is a regular connection on $E_G\big\vert_{X\setminus\{p\}}\, \longrightarrow\,
X\setminus\{p\}$. Also, $\operatorname{Ad}(g)(\theta)$
in \eqref{e32} is evidently defined over entire $X\setminus\{p\}$, because both $g$
and $\theta$ are defined over $X\setminus\{p\}$. From these it follows immediately that
the action of $L_XG$ on $T^*LG$ in \eqref{e32} preserves the subbundle $\mathcal W$.
\end{proof}

{}From Lemma \ref{lem6} we conclude that the action of $L_XG$ on $T^*LG$ in \eqref{e32}
induces an action on the subbundle $\mathcal W$ in
in \eqref{e17}. Therefore, using Corollary \ref{cor1}, we have the following:

\begin{corollary}\label{cor2}
Consider the action of $L_XG$ on $\mathcal W$ induced by the action of $L_XG$
on $T^*LG$ in \eqref{e32}. Using the identification of $\mathcal W$ with
$\mathcal U$ in Corollary \ref{cor1}, this action of $L_XG$ on $T^*LG$
produces an action of $L_XG$ on $\mathcal U$.
\end{corollary}

\begin{lemma}\label{lem7}
The action of $L_XG$ on $T^*LG$ in \eqref{e32} preserves the Liouville
symplectic form on $T^*LG$. 
\end{lemma} 

\begin{proof}
This is a straight--forward computation. It should be clarified that the Liouville
$1$--form (see Section \ref{sec:Liouvilleforgroups}) is not preserved by the action of $L_XG$, but its exterior derivative, namely
the Liouville symplectic form on $T^*LG$, is preserved by the action of $L_XG$. This follows
using the fact that any connection on a principal bundle on a smooth complex curve is
automatically integrable. (Any connection on a principal bundle on $X$
is integrable because $\bigwedge^2 \Omega^1_X\,=\, \bigwedge^2 K_X \,=\,0$.)
\end{proof}

Consider the action of $L_XG$ on $\mathcal U$ obtained in Corollary \ref{cor2}.
The corresponding quotient
\begin{equation}\label{30}
{\mathcal M}_{Conn}(G)\ :=\ \mathcal{U}/L_XG
\end{equation}
is the moduli stack of principal $G$--bundles $F_G$ on $X$ equipped with
\begin{itemize}
\item a trivialization of $F_G\big\vert_{D_p}\, \longrightarrow\, {D_p}$
on the formal disc $D_p$ around the point $p\, \in\, X$, and

\item a meromorphic connection on $F_G$.
\end{itemize}
The following theorem is an analog of Theorem \ref{thm:main1} for the moduli stack
$\mathcal{M}$ defined \eqref{30}.

\begin{theorem}\label{thm:main2}
The moduli stack ${\mathcal M}_{Conn}(G)$ in \eqref{30}, which 
parametrizes the principal $G$--bundles on $X$ with a meromorphic connection and
a trivialization over $D_p$, has a canonical symplectic structure coming from
Liouville symplectic structure on the total space of the cotangent bundle $T^*LG$. 
\end{theorem}

\begin{proof}
Just like Theorem \ref{thm:main1}, this theorem will also be proved using
Proposition \ref{prop:mainprop}. To get into the set--up of Proposition \ref{prop:mainprop},
set $\mathscr{G}$ in Proposition \ref{prop:mainprop} to be $L_XG$. Also, set
$M\,=\, LG$ in Proposition \ref{prop:mainprop}. Set $\omega$ in
Proposition \ref{prop:mainprop} to be the nondegenerate symmetric
bilinear form $\omega$ on $LG$ constructed in \eqref{e16}.

Define $\mathcal{F}^{\mathscr{G}}$ as in \eqref{e20}, and then define $\mathcal{V}$ as 
\eqref{e7a}. Then $\mathcal W$ coincides with $\mathcal{V}$, which
follow using Theorem \ref{thm:residueann}. Therefore,
$\mathcal U$ is identified with $\mathcal V$, because $\mathcal U$ is identified with
$\mathcal W$. Now the theorem follows from Proposition \ref{prop:mainprop}.
\end{proof}

\subsection{Singular connections with pole of bounded order}

As in Section \ref{se-hf}, let $k$ be a positive integer and consider principal $G$ bundles $E_G$
on $X$ which is equipped with
\begin{itemize}
\item a framing of order $k$ (see Definition \ref{def:kframing}) at the point $p \,\in\, X$, and

\item a meromorphic connection with a pole of order at most $k$ at the point $p$.
\end{itemize}
Recall the fiber bundle $\Phi\,:\, {\mathcal U}\, \longrightarrow\, LG$ constructed in \eqref{e31}.
Let $$\mathcal{U}^{k}\ \subset \ {\mathcal U}$$
be the subbundle whose fiber over any point $\alpha\, \in\,LG$ consists of all principal $G$--bundle $E_G$ on
$X$ equipped with a regular connection on $E_G\big\vert_{X\setminus\{p\}}$
whose pole at $p$ has order at most $k$.

As in Section \ref{sec:arbitraryframing}, the group $G^{k}[[t]]\times L_XG$ acts on the loop group 
$LG$ and hence also on $T^*LG$. Thus applying Proposition \ref{prop:mainprop}, we get the following: 

\begin{theorem}\label{thm:main3}
Let ${\mathcal M}_{Conn}^k(G)$ denote the moduli stack parametrizing the principal $G$--bundles on $X$ equipped
with a $k$--th order framing at the point $p$ and a meromorphic connection with pole of order at most
$k$ at $p$. Then ${\mathcal M}_{Conn}^k(G)$ inherits a canonical symplectic structure coming from the
symplectic structure on $T^*LG$. 
\end{theorem}
 
\section*{Acknowledgements}

S. Mukhopadhyay acknowledges support of the DAE, Government of India, under Project 
Identification No. RTI4001. I. Biswas is partially supported by a J. C. Bose Fellowship 
(JBR/2023/000003). M. Inaba was supported by JSPS KAKENHI 19K03422 and 22H00094. A. Komyo 
was supported by JSPS KAKENHI 19K14506, 24K06674 and 22H00094. M.-H. Saito was supported by 
JSPS KAKENHI 22H00094.

\end{document}